\documentclass{article}

\usepackage{amsmath}
\usepackage{epsfig}
\usepackage{amsfonts}
\usepackage{latexsym}
\usepackage{subfigure}
\usepackage{bbm}

\def\Op{\mathcal{Q}}
\def\Kcal{\mathcal{K}}
\def\Ordo{\mathcal{O}}
\def\Real{\mathbb{R}}  
\def\Cnumbers{\mathbb{C}}  

\DeclareMathOperator*{\conv}{conv}
\DeclareMathOperator*{\argmin}{argmin}
\DeclareMathOperator*{\supp}{supp}

\newcommand{\lbeq}[1]{{\label{OR:eq:#1}}}

\newcommand{\lbfig}[1]{{\label{OR:fig:#1}}}
\newcommand{\lbsec}[1]{{\label{OR:sec:#1}}}

\newcommand{\lbtheo}[1]{{\label{OR:theo:#1}}}
\newcommand{\lblem}[1]{{\label{OR:lem:#1}}}

\newcommand{\lbrem}[1]{{\label{OR:rem:#1}}}

\newcommand{\eq}[1]{{(\ref{OR:eq:#1})}}

\newcommand{\fig}[1]{{Figure~\ref{OR:fig:#1}}}
\newcommand{\sect}[1]{{Section~\ref{OR:sec:#1}}}

\newcommand{\theo}[1]{{Theorem~\ref{OR:theo:#1}}}
\newcommand{\lem}[1]{{Lemma~\ref{OR:lem:#1}}}

\newcommand{\rem}[1]{{Remark~\ref{OR:rem:#1}}}

\newcommand{\be}[1]{\begin{equation} \lbeq{#1}}
\newcommand{\ee}{\end{equation}}
\newcommand{\bea}[1]{\begin{eqnarray} \lbeq{#1}}
\newcommand{\eea}{\end{eqnarray}}
\newcommand{\bear}{\begin{array}}
\newcommand{\eear}{\end{array}}

\newtheorem{remark}{Remark}[section]

\newtheorem{theorem}{Theorem}

\newtheorem{lemma}{Lemma}

\newenvironment{proof}{\textit{Proof:}\ }{$~\Box$}

\title{Error Estimates for Gaussian Beam Superpositions}
\author{Hailiang Liu\thanks{Department of Mathematics, Iowa State University, Ames, IA 50011, USA. ({\tt hliu@iastate.edu}).} \and Olof Runborg\thanks{Department of Numerical Analysis, CSC, KTH, 100 44 Stockholm, Sweden and Swedish e-Science Research Center (SeRC), KTH, 100 44 Stockholm, Sweden. ({\tt olofr@nada.kth.se}).} \and Nicolay M. Tanushev\thanks{Department of Mathematics, The University of Texas at Austin, 1 University Station, C1200, Austin, TX 78712, USA. ({\tt nicktan@math.utexas.edu}).} }

\begin{document}
\maketitle

\begin{abstract}
Gaussian beams are asymptotically valid high frequency solutions to hyperbolic
partial differential equations, concentrated on a single curve through the
physical domain. They can also be extended to some dispersive wave equations,
such as the Schr\"odinger equation. Superpositions of Gaussian beams provide a
powerful tool to generate more general high frequency solutions that are not
necessarily concentrated on a single curve. This work is concerned with the
accuracy of Gaussian beam superpositions in terms of the wavelength
$\varepsilon$. We present a systematic construction of Gaussian beam
superpositions for all strictly hyperbolic and Schr\"odinger equations subject
to highly oscillatory initial data of the form $Ae^{i\Phi/\varepsilon}$. Through
a careful estimate of an oscillatory integral operator, we prove that the $k$-th
order Gaussian beam superposition converges to the original wave field at a rate
proportional to $\varepsilon^{k/2}$ in the appropriate norm dictated by the
well-posedness estimate. In particular, we prove that the Gaussian beam
superposition converges at this rate for the acoustic wave equation in the
standard, $\varepsilon$-scaled, energy norm and for the Schr\"odinger equation
in the $L^2$ norm. The obtained results are valid for any number of spatial
dimensions and are unaffected by the presence of caustics. We present a
numerical study of convergence for the constant coefficient acoustic wave
equation in $\Real^2$ to analyze the sharpness of the theoretical results.
\end{abstract}

%

\pagestyle{myheadings}
\thispagestyle{plain}
\markboth{H. LIU, O. RUNBORG AND N. M. TANUSHEV}{ERROR ESTIMATES FOR GAUSSIAN BEAM SUPERPOSITIONS}


\section{Introduction}
In simulations of high frequency wave propagation, a large number of grid points is
needed to resolve and maintain an accurate in time representation of the wave field.
Consequently, in this regime, direct numerical simulations are computationally
expensive and at sufficiently high frequencies, such simulations are no longer
feasible. To circumvent this difficulty, approximate high frequency
asymptotically valid methods are often used. One such popular approach is
geometrical optics \cite{EngRun:03, Runborg:07}, which is obtained in the limit
when the frequency tends to infinity. This method is also known as the WKB
method or ray-tracing. The solution of the partial differential equation (PDE)
is assumed to be of the form
\begin{align}\lbeq{GOform}
a(t,y,\varepsilon) e^{i\phi(t,y)/\varepsilon},
\end{align}
where $1/\varepsilon$ is the large high frequency parameter, $\phi$ is the phase, and $a$ is the
amplitude of the solution having the Debye expansion in terms of $\varepsilon$, $a(t,y,\varepsilon)=\sum_{j=0}^N \varepsilon^j a_j(t,y)$. The phase and amplitudes $a_j$ are
independent of the frequency and vary on a much coarser scale than the full wave
solution. They can therefore be computed at a computational cost independent of
the frequency. However, the geometrical optics approximation breaks down at
caustics, where rays concentrate and the predicted amplitude is unbounded
\cite{Ludwig:66, Kravtsov:64}. The consideration of difficulties caused by
caustics, beginning with Keller in \cite{Keller:1958} and Maslov and Fedoriuk (see
\cite{MaslovFedoriuk:1981}), led to the development of the theory of Fourier
integral operators, e.g., as given by H\"ormander in \cite{HormanderFIO:1971}.

Gaussian beams form another high frequency asymptotic model which is closely
related to geometrical optics. However, unlike geometrical optics, Gaussian
beams do not breakdown at caustics. For Gaussian beams, the solution is also
assumed to be of the geometrical optics form \eq{GOform}, but a Gaussian beam is
a localized solution that concentrates near a single ray of geometrical
optics in space-time. Although the phase function is real-valued along the
central ray, Gaussian beams have a complex-valued phase function off their
central ray. The imaginary part of the phase is chosen such that the solution
decays exponentially away from the central ray, maintaining a Gaussian-shaped
profile. To form a Gaussian beam solution, we first pick a ray and solve a
system of ordinary differential equations (ODEs) along it to find the values of
the phase, its first and second order derivatives and the amplitude on the ray.
To define the phase and amplitude away from this ray to all of space-time, we
extend them using a Taylor expansion. Heuristically speaking, along each ray we
propagate information about the phase and amplitude and their derivatives that
allows us to reconstruct the wave field locally in a Gaussian envelope. The
existence of Gaussian beam solutions has been known since sometime in the 1960's,
first in connection with lasers, see Babi\v{c} and Buldyrev
\cite{BabicBuldyrev:1991}. Later, they were used in the analysis of propagation
of singularities in partial differential equations by H\"ormander
\cite{Hormander:71} and Ralston \cite{Ralston:82}.

In this article, we are interested in the accuracy of Gaussian beam solutions to
$m$-th order linear, strictly hyperbolic PDEs with highly oscillatory initial
data of the type
\begin{align}\lbeq{generalHyperbolicEqIntro}
 Pu &= 0, \qquad (t,y) \in (0,T]\times\Real^n, \\ 
 \partial_t^\ell u(0,y) &= \varepsilon^{-{\ell}}\sum_{j=0}^N \varepsilon^j A_{{\ell},j}(y)e^{i\Phi(y)/\varepsilon} \ , \qquad {\ell}=0,\ldots,m-1 \ ,  \notag
\end{align}
where the strictly hyperbolic operator, $P$, is defined in
\sect{GBforHyperbolicEq}, the real valued phase, $\Phi(y)$ belongs to
$C^\infty(K_0;\Real)$ for some compact set $K_0\subset \Real^n$, and the complex
valued amplitudes, $A_{{\ell},j}(y)$ belong to $C^\infty_0(K_0;\Cnumbers)$.
Furthermore, we will assume that $|\nabla\Phi(y)|$ is bounded away from zero on
$K_0$. As a special case, we include the acoustic wave equation,
\begin{gather}
  u_{tt} - c(y)^2\Delta u = 0, \qquad (t,y)\in(0,T]\times\Real^n \ , \notag \\
  u(0,y)   = \sum_{j=0}^N \varepsilon^j A_{0,j}(y)e^{i\Phi(y)/\varepsilon} \ , \quad\mbox{and}\quad
  u_t(0,y) = \frac{1}{\varepsilon}\sum_{j=0}^N \varepsilon^j A_{1,j}(y)e^{i\Phi(y)/\varepsilon} \ .\lbeq{WaveEquationIntro}
\end{gather}

We also treat the dispersive Schr\"odinger equation,
\begin{align}\lbeq{SchrodingerEquationIntro}
  -i\varepsilon u_t - \frac{\varepsilon^2}{2}\Delta u + V(y)u &= 0, \qquad (t,y)\in (0,T]\times \Real^n, \\
  u(0,y) &=\sum_{j=0}^N \varepsilon^j A_j(y)e^{i\Phi(y)/\varepsilon} \notag \ .
\end{align}
As above, we will assume that $\Phi\in C^\infty(K_0;\Real)$ and $A_j\in
C^\infty_0(K_0;\Cnumbers)$. Furthermore, we assume the potential $V(y)$ is
smooth and bounded along with all its derivatives, $\partial_y^\beta V\in
C^\infty_b(\Real^n)$ for all $\beta$. For the Schr\"odinger equation, the
asymptotic parameter $\varepsilon$ appears in the equation and there is no need
to assume the bound on $|\nabla\Phi|$ that is necessary for strictly hyperbolic
equations.

Since these partial differential equations are linear, it is a natural extension
to consider sums of Gaussian beams to represent more general high frequency
solutions that are not necessarily concentrated on a single ray. This idea was
first introduced by Babi\v{c} and Pankratova in \cite{Babic:1973} and was later
proposed as a method for wave propagation by Popov in \cite{Popov:1982}. The sum, or rather the integral superposition, of Gaussian beams in the simplest first
order form can be written as
\begin{align}\lbeq{GBform1}
   u_{GB}(t,y) =\left(\frac{1}{2\pi\varepsilon}\right)^\frac{n}{2} 
   \int_{K_0} a(t;z)  e^{i\phi(t,y-x(t;z);z)/\varepsilon} dz \ ,
\end{align}
where $K_0$ is a compact subset of $\Real^n$ and the phase that defines the
Gaussian beam is given by
\begin{align}\lbeq{GBform2}
\phi(t,y;z)=\phi_{0}(t;z) + y\cdot p(t;z) + y\cdot \frac12 M(t;z) y\ .
\end{align}
The real vector $p(t;z)$ is the direction of wave propagation and the matrix
$M(t;z)$ has a positive definite imaginary part and it gives Gaussian beams
their profile. Extensions of the above superposition are possible in several
directions, including using higher order Gaussian beams in the superposition and
using a sum of several superpositions to approximate the different modes of wave
propagation. Higher order Gaussian beams are created by using an asymptotic
series for the amplitude and using higher order Taylor expansions to define
the phase and the amplitude functions, see \eq{PhaseWithZ} and \eq{AmpWithZ}.
Also, for higher order beams, a cutoff function \eq{cutoff} is necessary to
avoid spurious growth away from the central ray. Superpositions with higher order
Gaussian beams have an improved asymptotic convergence rate. For $m$-th order
strictly hyperbolic PDEs, which have $m$ pieces of initial data, we use $m$
different Gaussian beam superpositions chosen in such a way so that their sum
approximates the initial data. Each of these superpositions corresponds to one
of the $m$ distinct modes of wave propagation.

Accuracy studies for a Gaussian beam solution $u_{\rm GB}$ have traditionally
focused on how well it asymptotically satisfies the PDE, i.e. the size of the
norm of $Pu_{\rm GB}$ in terms of $\varepsilon$. The question of determining the
error of the Gaussian beam superposition compared to the exact solution was
thought to be a rather difficult problem decades ago, see the conclusion section
of the review article by Babi\v{c} and Popov \cite{BabicPopov:1989}. However,
some progress on estimates of the error has been made in the past few years.
This accuracy study was initiated by Tanushev in \cite{Tanushev:08}, where a
convergence rate was obtained for the initial data. Some earlier results on this
were also established by Klime\v{s} in \cite{Klimes:86}. The part of the error
that is due to the Taylor expansion off the central ray was considered by
Motamed and Runborg in \cite{MotamedRunborg:09} for the Helmholtz equation. Liu
and Ralston \cite{LiuRalston:09,LiuRalston:10} gave rigorous convergence rates
in terms of $\varepsilon$ for the acoustic wave equation in the scaled energy
norm and for the Schr\"odinger equation in the $L^2$ norm. However, the error
estimates they obtained depend on the number of space dimensions in the presence
of caustics, since the projected Hamiltonian flow to physical space becomes
singular at caustics. The superpositions in \eq{GBform1} can also be carried out
over both $x$ and $p$ in full phase space through the Hamiltonian map,
$(z,p_0)\to(x(t;z,p_0),p(t;z,p_0))$, as shown in
\cite{LiuRalston:09,LiuRalston:10}. In this formulation the Hamiltonian flow is
regular and there are no caustics, so we expect to obtain a dimensionally
independent error estimate. This has been confirmed for the wave equation by
Bougacha, Akian and Alexandre in \cite{BougachaEtal:09} for the case of initial
data based on the Fourier--Bros--Iaglonitzer (FBI) transform, a result which is
inspired by the work of Rousse and Swart on the Herman-Kluk propagator for the
Schr\"odinger equation \cite{RousseSwart:07, RousseSwart:09}. From a
computational stand point, the full phase space formulation is more expensive.

Building upon these recent advances, together with an application of a non-squeezing
argument proved in \lem{nonsqueezing}, we are able to provide a definite answer to the
question of accuracy for Gaussian beam superposition solutions. More precisely,
we obtain dimensionally independent estimates for the superposition in physical
space for general $m$-th order strictly hyperbolic PDEs and the Schr\"odinger
equation. Our main result is the following theorem.

\begin{theorem}\lbtheo{ErrorEstWaveSchrod}
Let $u$ be the exact solution to the PDEs considered, 
\eq{generalHyperbolicEqIntro}, \eq{WaveEquationIntro}
and \eq{SchrodingerEquationIntro}, under the stated assumptions
on initial data and partial differential operators.
Moreover, let $u_k$ be the corresponding $k$-th order Gaussian beam
superposition given in \sect{SuperpositionsOfGB}
and \sect{SchrodingerEq}, with a sufficiently
small cutoff parameter $\eta$ when $k>1$.
We then have the following
estimates. 
For the $m$-th order strictly hyperbolic PDE \eq{generalHyperbolicEqIntro},
\begin{align*}
 \varepsilon^{m-1}\sum_{\ell=0}^{m-1}\left\|\partial^{\ell}_t [u(t,\cdot)-u_k(t,\cdot)]\right\|_{H^{m-\ell-1}} \leq C(T)\varepsilon^{k/2}\ .
\end{align*}
For the acoustic wave equation  \eq{WaveEquationIntro},
\begin{align*}
 ||u(t,\cdot)-u_k(t,\cdot)||_{E} \leq C(T)\varepsilon^{k/2}\ ,
\end{align*}
where $||\,\cdot\,||_E$ is the scaled energy norm \eq{ENorm}.
Finally, for the Schr\"odinger equation \eq{SchrodingerEquationIntro},
\begin{align*}
 ||u(t,\cdot)-u_k(t,\cdot)||_{L^2} \leq C(T)\varepsilon^{k/2}\ .
\end{align*}
\end{theorem}

This improves on the results in \cite{LiuRalston:09,LiuRalston:10} where
the last two error estimates were also proved, but with an additional
factor $\varepsilon^{-\gamma}$ in the right hand side, where $\gamma=(n-1)/4$ for the
wave equation and $\gamma=n/4$ for the Schr\"odinger equation.
Note that the rescaling by $\varepsilon^{m-1}$ in the first estimate
is convenient here, since it exactly
balances the rate at which the corresponding norm of the initial data for the PDE
\eq{generalHyperbolicEqIntro} goes to infinity as $\varepsilon\to 0$.

At present there is considerable interest in using numerical methods based on
superpositions of beams to resolve high frequency waves near caustics, which
began in the 1980's with numerical methods for wave propagation in
\cite{Popov:1982,Katchalov_Popov:1981,Cerveny_etal:1982} and more specifically
in geophysical applications in \cite{Klimes:1984,Hill:1990,Hill:2001}. Recent
work in this direction includes simulations of gravity waves \cite{TQR:2007}, of
the semiclassical Schr\"{o}dinger equation \cite{JinWuYang:08,LeungQian:09}, and
of acoustic wave equations \cite{MotRun_GB1,Tanushev:08}. Numerical techniques
based on both Lagrangian and Eulerian formulations of the problem have been
devised \cite{JinWuYang:08,LeuQiaBur:07,LeungQian:09,MotRun_GB1}. A numerical
approach for general high frequency initial data closely related to the FBI
transform, but avoiding the cost of superposing over all of phase space, is presented
in \cite{QianYing:2010} for the Schr\"odinger equation. Numerical approaches for
treating general high frequency initial data for superposition over physical
space were considered in \cite{TET:2009,AETT:2010} for the wave equation. Our
theoretical results show that the numerical solutions found in these papers will
be accurate when $\varepsilon\ll 1$.
 
To test the sharpness of the theoretical convergence rates, we present a short
numerical study in the case of the acoustic wave equation with constant sound
speed. Our numerical results indicate that the theoretical rates are sharp for
even order $k$, but similar to the result in \cite{MotamedRunborg:09}, we
observe a gain in the convergence rate of a factor of $\varepsilon^{1/2}$ for beams of
odd order $k$, which suggests that the actual convergence rate is
$\Ordo(\varepsilon^{\lceil k/2\rceil})$.

This paper is organized as follows: \sect{GaussianBeams} introduces Gaussian
beams and their superpositions for $m$-th order strictly hyperbolic equations.
Furthermore, we construct Gaussian beams for the Schr\"odinger equation.
\sect{ErrorEstForGB} is devoted to error estimates for Gaussian beam
superpositions. Detailed norm estimates of the oscillatory operators used in
obtaining the error estimates are given in \sect{NormEstQ}. Numerical validation
of our results is finally presented in \sect{Numerical}.

\section{Construction of Gaussian Beams} \lbsec{GaussianBeams}

In this section, we outline the construction of Gaussian beam superpositions for
strictly hyperbolic PDEs. We also construct Gaussian beams for the Schr\"odinger
equation.

\subsection{Hyperbolic Equations}\lbsec{GBforHyperbolicEq}

Let $P=P_m+L$ be a linear strictly hyperbolic $m$-th order partial differential
operator (PDO) in $n$ dimensions with
\be{PHypDef}
   P_m = \partial^m_t + \sum_{j=0}^{m-1} \left(\sum_{|\beta|=m-j}g_{\beta}(t,y)\partial_y^\beta\right)\partial_t^j \ ,
\ee
and $L$ a differential operator of order $m-1$. The principal symbol of $P$, denoted by $\sigma_m(t,y,\tau,p)$,
is defined by the formal relationship $P_m = \sigma_m(t,y,-i\partial_t,-i\partial_y)$.
Following \cite{HormanderIII,Ralston:82}, we make the assumptions:
\begin{enumerate}
\renewcommand{\theenumi}{(H\arabic{enumi})}
\renewcommand{\labelenumi}{(H\arabic{enumi})}
\item\label{HypAsmptSmooth} The coefficients $g_{\beta}(t,y)$ are smooth
functions, bounded in $t$ and $y$ along with all their derivatives,
$\partial_t^\ell\partial_y^\alpha g_{\beta}\in C^\infty_b(\Real^n)$ for all
$\ell,\alpha$.
\item\label{HypAsmptRoots} For $|p|\neq 0$, the principal symbol
$\sigma_m(t,y,\tau,p)$ has $m$ distinct real roots, when it is considered as a
polynomial in $\tau$.
\item\label{HypAsmptUnifSimple} These roots are uniformly simple in the sense
that
\be{uniformlydistinct}
\left|\frac{\partial\sigma_m(t,y,\tau,p)}{\partial\tau}\right|\geq c_0 |p|^{m-1}
\quad {\rm whenever}\quad \sigma_m(t,y,\tau,p)=0.
\ee
\end{enumerate}

We consider a null bicharacteristic $(t(s),x(s),\tau(s),p(s))$ associated with the principal symbol $\sigma_m$, defined by the Hamiltonian system of ODEs:
\begin{align}\lbeq{HamiltonianFlow}
  \dot{t}    &= \frac{\partial \sigma_m}{\partial\tau}\ ,
 &\dot{x}    &= \frac{\partial \sigma_m}{\partial p}\ ,
 &\dot{\tau} &= -\frac{\partial \sigma_m}{\partial t} \ ,
 &\dot{p}  &= -\frac{\partial \sigma_m}{\partial y} \ ,
\end{align}
and initial conditions $(t(0),x(0),\tau(0),p(0))$ such that
$\sigma_m(t(0),x(0),\tau(0),p(0))=0$, with $p(0)\neq 0$. Note that for fixed
$t(0)$, $x(0)$ and $p(0)$, we have $m$ distinct choices for $\tau(0)$, equal to
the $m$ distinct real roots of $\sigma_m(t(0),x(0),\tau,p(0))$. These choices
for $\tau(0)$ give $m$ distinct waves that travel in different directions. The
curve $(t(s),x(s))$ in physical space is the space-time ray that Gaussian beams
are concentrated near. For a proof that the Gaussian beam construction is only
possible near this ray, we refer the reader to $\cite{Ralston:82}$. 

The following lemma summarizes some results related to the Hamiltonian flow
above. We will use the second point to argue that changing variables $s\to t$ is
always allowed. The last point is needed in the proof of the non-squeezing lemma
(\lem{nonsqueezing}).
\begin{lemma}\lblem{HyperbolicHamiltonianLemma}
Let $(t(s),x(s),\tau(s),p(s))$ be a null bicharacteristic of the Hamiltonian flow
\eq{HamiltonianFlow} with initial data such that $|p(0)|\neq 0$ and
$\sigma_m(t(0),x(0),\tau(0),p(0))=0$. Without loss of generality, assume that the parametrization is taken so that $\dot{t}(0)\geq 0$. Then for $s\in[0,\infty)$, we have
\begin{enumerate}
 \item $\sigma_m(s) = \sigma_m(t(s),x(s),\tau(s),p(s)) = 0$,
 \item $t(s)-t(0)$ is strictly increasing and
 $$
    t(s)-t(0) \geq C_1 \begin{cases} s\ , & m=1\ , \\
    \log\left(1+C_2|p(0)|^{m-1}\ s \right)\ , & m>1\ ,
    \end{cases}
 $$
 where $C_1$ and $C_2$ are independent of $s$ and initial data.
 \item There is a constant $\lambda $ independent of $s$ and initial data such that
 \begin{align}\lbeq{hamres3}
 |p(s)|\geq |p(0)|e^{-\lambda (t(s)-t(0))}.
 \end{align}
\end{enumerate}
\end{lemma}

\begin{proof}
We compute
\begin{align*}
 \dot{\sigma}_m(s) & = \frac{\partial\sigma_m}{\partial t} \dot{t} + \frac{\partial\sigma_m}{\partial y} \cdot \dot{x} + \frac{\partial\sigma_m}{\partial \tau} \dot{\tau} + \frac{\partial\sigma_m}{\partial p} \cdot \dot{p}\\
& = \frac{\partial\sigma_m}{\partial t} \frac{\partial \sigma_m}{\partial\tau} + \frac{\partial\sigma_m}{\partial y} \cdot \frac{\partial \sigma_m}{\partial p} - \frac{\partial\sigma_m}{\partial \tau} \frac{\partial \sigma_m}{\partial t} - \frac{\partial\sigma_m}{\partial p} \cdot \frac{\partial \sigma_m}{\partial y} = 0,
\end{align*}
and since $\sigma_m(0)=0$, we have that $\sigma_m(s)=0$, proving the first
point.

Since $\sigma(x,t,p,\tau)=0$, we note that $|p(s_0)|=0$ implies that $\tau(s_0)=0$. Thus, if
$|p(s_0)|=0$ then $\dot{p}(s_0)=\frac{\partial\sigma_m}{\partial y}(s_0)=0$.
Hence, we have that for all $s\in[0,\infty)$, $|p(s)|\neq 0$ by uniqueness for
solutions of ODEs and $|p(0)|\neq 0$. Recall that for a polynomial with distinct
root, both the polynomial and its derivative cannot vanish at the same point.
Since $|p(s)|\neq 0$ and strict hyperbolicity imply that
$\sigma_m(t(s),x(s),\tau,p(s))$, as a polynomial of $\tau$, has distinct roots
and $\sigma_m=0$ on the null bicharacteristic, we have that
$\frac{\partial\sigma_m}{\partial\tau}(s)\neq 0$. Continuity implies that
$\dot{t}(s) = \frac{\partial\sigma_m}{\partial\tau}(s)$ never changes sign and,
by the choice in parametrization, we have that $\dot{t}(s)>0$. Hence, $t(s)-t(0)$ is strictly
increasing for all $s\in[0,\infty)$.

We next show that there is a constant $C$ such that $|\tau(s)|\leq C |p(s)|$ for all $s$
and initial data. This is obviously true if $\tau=0$. When $\tau\neq 0$,
let $\xi = p/\tau$ and observe that, by homogeneity,
$$
  0 = \sigma_m(t,y,\tau,p) = \tau^m\sigma_m(t,y,1,\xi)\ ,
$$
on the null bicharacteristic. Hence $\xi$ is a root of the polynomial equation
$$
1+\sum_{|\beta|=1}^m g_{\beta}(t,y)\xi^\beta  = 0\ .
$$
Since the coefficients $g_{\beta}(t,y)$ are bounded it follows that there is a
constant such that $|\xi|\geq C >0$, which proves that $|\tau(s)|\leq C |p(s)|$
for all $s$.

We can now bound $\dot{p}(s)$ as follows.
\begin{align*}
  \left|\dot{p}\right| &= 
    \left|\frac{\partial\sigma_m}{\partial y}\right| = 
  \left|\sum_{j=0}^{m-1} \left(\sum_{|\beta|=m-j}\partial_y g_{\beta}(t,y)p^\beta\right)\tau^j
  \right|\\
  &\leq 
  C\sum_{j=0}^{m-1} \sum_{|\beta|=m-j} |p|^{|\beta|}|\tau|^j=
  C\sum_{j=0}^{m-1} \sum_{|\beta|=m-j} |p|^m |\xi|^{-j}\leq c_1 |p|^m\ ,
\end{align*}
for some constant $c_1$ independent of $s$ and initial data.
Let $\lambda = c_1/c_0$ where $c_0$ is the constant in \eq{uniformlydistinct}.
Moreover, set $\tilde{t}(s) = t(s)-t(0)\geq 0$ and note that $\dot{\tilde t}(s)=\dot{t}(s)>0$. Then, 
\begin{align*}
   \frac{d}{ds} |p(s)|^2e^{2\lambda \tilde{t}(s)} &= 2\; \dot{p}\cdot p\;e^{2\lambda \tilde{t}}+ 2\lambda \dot{t}|p|^2e^{2\lambda \tilde{t}} 
    \geq -2|\dot{p}| |p|e^{2\lambda \tilde{t}}+ 
    2\lambda \frac{\partial\sigma_m}{\partial\tau}|p|^2e^{2\lambda \tilde{t}}
    \\&\geq
    -2c_1|p|^{m+1}e^{2\lambda \tilde{t}}+ 2\lambda c_0|p|^{m+1}e^{2\lambda \tilde{t}} =0\ .
\end{align*}
Consequently, $|p(s)|^2e^{2\lambda \tilde{t}(s)}\geq |p(0)|^2$ 
and \eq{hamres3} follows. With $\lambda_m = (m-1)\lambda$, we obtain
$$
  \dot{t} = 
  \frac{\partial\sigma_m(t,y,\tau,p)}{\partial\tau}\geq c_0 |p|^{m-1}
  \geq c_0|p(0)|^{m-1} e^{-\lambda_m \tilde{t}(s)}\ .
$$
It follows that  $\tilde{t}(s)\geq c_0 s$ for $m=1$, since $\lambda_1=0$. For $m>1$,
\begin{align*}
  s c_0|p(0)|^{m-1} &\leq \int_0^s \dot{t}(s') e^{\lambda_m \tilde{t}(s')}ds'
  = \int_0^s\dot{\tilde{t}}(s')e^{\lambda_m \tilde{t}(s')}ds'\\
  &=\int_{0}^{\tilde{t}(s)} e^{\lambda_m \theta }d\theta
   = \frac{e^{\lambda_m \tilde{t}(s)}-1}{\lambda_m} \ .
\end{align*}
This proves the stated logarithmic growth of $\tilde{t}(s)$ for $m>1$.
\end{proof}

As mentioned above, an immediate consequence of this lemma is that we can use
$t$ to parametrize the Hamiltonian flow \eq{HamiltonianFlow} instead of $s$,
since the lemma guarantees that for a fixed $t_0\in[0,T]$, there exists a unique
$s_0$ such that $t(s_0)=t_0$. With a slight abuse of notation we will now write
$x(t)$ and $p(t)$ for the Hamiltonian flow parametrized by $t$.

Following \cite{Ralston:82}, we define a phase function $\phi$ and amplitude
functions $a_j$ via Taylor polynomials. After changing variables $s\to t$ in the
formulation used in \cite{Ralston:82} we can write:
\begin{align}
  \phi(t,y) &= \sum_{|\beta|=0}^{k+1} \frac1{\beta!}\phi_\beta(t) y^\beta  \equiv \phi_0(t) + y\cdot p(t) + y\cdot \frac12 M(t) y + \sum_{|\beta|=3}^{k+1} \frac1{\beta!}\phi_\beta(t) y^\beta \ , \nonumber\\
  a_j(t,y)  &= \sum_{|\beta|=0}^{k-2j-1} \frac1{\beta!}a_{j,\beta}(t) y^\beta \ .
  \lbeq{Taylorexpansion}
\end{align}
We can now define the preliminary $k$-th order Gaussian beam $\tilde{v}_k(t,y)$
as:
\begin{align*}
   \tilde{v}_k(t,y) = \sum_{j=0}^{\lceil k/2 \rceil -1} \varepsilon^j a_j(t,y-x(t))  e^{i\phi(t,y-x(t))/\varepsilon}\ .
\end{align*}
Applying the operator $P$ to this beam and collecting terms containing the same
power of $\varepsilon$, we have
\begin{align*}
 P\tilde{v}_k(t,y) &= \left(\sum_{r=-m}^{J} \varepsilon^r c_r(t,y) \right)e^{i\phi(t,y-{x}(t))/\varepsilon}\ ,
\end{align*}
where $c_r(t,y)$ are smooth functions independent of $\varepsilon$. The
construction in \cite{Ralston:82} then proceeds to make $c_r(t,y)$ vanishe up to
order $k-2(r+m)+1$ on $y=x(t)$. To obtain this, $(x(t),p(t))$ must follow the
Hamiltonian flow and the coefficients in the Taylor polynomials should satisfy
ODEs, which are given as follows. By assumption \ref{HypAsmptRoots} above,
whenever $p\neq 0$, we can define $m$ Hamiltonians $H_\ell(t,x,p)$ implicitly by
the relations $\sigma_m(t,x,-H_\ell(t,x,p),p)=0$ for $\ell=0,\ldots,m-1$. For
any choice $H=H_\ell$, the first several ODEs are
\begin{align}\lbeq{Hsystem}
 \dot{x} &=  \partial_p H(t,x,p) \ , &  \dot{p} &= -\partial_y H(t,x,p) \ , \\
 \dot{\phi}_0 &=-H + p\cdot\partial_p H(t,x,p) \ , &
\dot{M} &= -{A} - {M}{B} - {B}^{\sf T}{M} - {M}C{M} \ ,\notag
\end{align} 
with 
\begin{align*}
{A} = \frac{\partial^2 H}{\partial y^2}\ , \qquad
 {B} = \frac{\partial^2H}{\partial {p}\partial y}  \ ,\qquad
{C} = \frac{\partial^2 H}{\partial p^2}  \ .
\end{align*}
By \eq{hamres3} the Hamitonian will be well-defined for all times if $p(0)\neq
0$. Moreover, we have the following result for the Hessian matrix $M(t)$ of the
phase that guarantees that the leading order shape of the beam stays Gaussian
for all time:
\begin{lemma}[Ralston '82, \cite{Ralston:82}]\lblem{HessianLemma}
Suppose that $M(0)$ is chosen so that it has a positive definite
imaginary part, $M(0)\dot{x}(0)=\dot{p}(0)$, and
$M(0)=M^{\sf T}(0)$. Then, for all $t\in[0,T]$, $M(t)$ 
will be such that $M(t)$ will have a positive definite imaginary part and $M(t)=M^{\sf T}(t)$.
\end{lemma}

Before we can fully define a Gaussian beam, one last point needs to be
addressed. Since $x(t)$ and $p(t)$ are real, if $\phi_0(0)$ is real and
$M(0)$ chosen as in \lem{HessianLemma}, then the imaginary part of $\phi(t,y)$
will be a positive quadratic plus higher order terms about $x(t)$. Thus, we must
only construct the Gaussian beam in a domain on which the quadratic part is
dominant. To this end, we use a cutoff function $\rho_\eta\in
C^\infty(\Real^n;\Real)$ with cutoff radius $0<\eta\leq\infty$ satisfying,
\begin{align}\lbeq{cutoff}
\rho_\eta(z) \geq 0 \quad \mbox{ and } \quad \rho_\eta(z)= \left\{ \begin{array}{ll}\begin{array}{l} 1 \mbox{ for } |z|\leq\eta \\ 0 \mbox{ for } |z|\geq 2\eta \end{array} &\mbox{ for } 0<\eta<\infty \\ \\\begin{array}{l}1\end{array} &\mbox{ for } \eta=\infty \end{array}\right. \ .
\end{align}
Now, by choosing $\eta>0$ sufficiently small, we can ensure that on the support
of $\rho_\eta(y-x(t))$, $\Im\phi(t,y-x(t))>\delta|y-x(t)|^2$ for $t\in[0,T]$.
However, note that for first order beams the imaginary part of the phase is
quadratic with no higher order terms, so that the cutoff is unnecessary. Thus,
to include this case we let the cutoff function be defined for $\eta=\infty$ by
$\rho_\infty\equiv 1$. We are now ready to finally define the $k$-th order
Gaussian beam $v_k(t,y)$ as:
\begin{align}\lbeq{SingleGBExpansion}
   v_k(t,y) = \sum_{j=0}^{\lceil k/2 \rceil -1} \varepsilon^j \rho_\eta(y-x(t))a_j(t,y-x(t))  e^{i\phi(t,y-x(t))/\varepsilon}\ .
\end{align}

\subsection{Superpositions of Gaussian Beams}\lbsec{SuperpositionsOfGB}

In the previous section, we introduced Gaussian beam solutions that satisfy
$Pu=0$ in an asymptotic sense (see \cite{Ralston:82} for a precise
statement), without too much concern for the values of the solution at $t=0$.
The initial value problem that a single Gaussian beam $v_k(t,y)$ approximates has
initial data that is simply given by its values (and the values of its time
derivatives) at $t=0$. While they resemble the initial conditions for
\eq{generalHyperbolicEqIntro}, they are quite different since for example the
phase, $\phi$, for $v_k$ is complex valued and $v_k$ is concentrated in $y$.

Our goal is to create an asymptotically valid solution to the PDE
\eq{generalHyperbolicEqIntro}, thus we must also consider the $m$ distinct
pieces of initial data given in the form of time derivatives of $u$ at $t=0$. To
generate solutions based on Gaussian beams that approximate the initial data for
\eq{generalHyperbolicEqIntro}, we exploit the linearity properties of $P$. That
is, we use the fact that a linear combination of two Gaussian beams, with
different initial parameters, will also be an asymptotic solution to $Pu=0$,
since each Gaussian beam is itself an asymptotic solution. Building on
this idea, we take a family of Gaussian beams that is indexed by a parameter
$z\in K_0$, where $K_0$ is the compact subset of $\Real^n$ discussed in the
introduction that contains the support of the amplitudes of the initial data
for \eq{generalHyperbolicEqIntro}.
We will use the notation $x_\ell(t;z)$,
$p_\ell(t;z)$, $\phi_\ell(t,y-x(t;z);z)$, etc, to denote the dependence of these
quantities on the indexing parameter $z$ and on the $m$ choices for
$H(t,x,p)=H_{\ell}(t,x,p)$ denoted by $\ell=0,\ldots,m-1$. Thus, we will write
the $k$-th order Gaussian beam $v_{k,\ell}(t,y;z)$. Note that the cutoff radius
$\eta$ may vary with $z$ and $\ell$ between beams, however, as the beam
superposition is taken over the compact set $K_0$ and $0\leq {\ell}\leq m-1$,
there is a minimum value for $\eta$ that will work for all beams in the
superposition. Thus, we form the $k$-th order superposition solution $u_k(t,y)$
as
\begin{align}\lbeq{superpositionForHyperbolic}
   u_k(t,y) = \sum_{{\ell}=0}^{m-1}\left(\frac{1}{2\pi\varepsilon}\right)^\frac{n}{2} \int_{K_0} v_{k,\ell}(t,y;z)dz \ ,
\end{align}
where the phase and amplitude that define the Gaussian beam,
\begin{align*}
 v_{k,\ell}(t,y;z) &=  \sum_{j=0}^{\lceil k/2 \rceil -1} \varepsilon^j \rho_\eta(y-x_{\ell}(t;z))a_{{\ell},j}(t,y-x_{\ell}(t;z);z)  e^{i\phi_{\ell}(t,y-x_{\ell}(t;z);z)/\varepsilon}\ ,
\end{align*}
are given by 
\begin{align}\lbeq{PhaseWithZ}
\phi_{\ell}(t,y;z)&=\phi_{{\ell},0}(t;z) + y\cdot p_{\ell}(t;z) + y\cdot \frac12 M_{\ell}(t;z) y + \sum_{|\beta|=3}^{k+1}\frac{1}{\beta!} \phi_{{\ell},\beta}(t;z) y^\beta \ ,
\\
\lbeq{AmpWithZ}
  a_{{\ell},j}(t,y;z)  &= \sum_{|\beta|=0}^{k-2j-1} \frac{1}{\beta!}a_{{\ell},j,\beta}(t;z) y^\beta \ .
\end{align}
We remind the reader that each $v_{k,\ell}(t,y;z)$ requires initial values
for the ray and all of the amplitude and phase Taylor coefficients. The
appropriate choice of these initial values will make $u_k(0,y)$
asymptotically converge the initial conditions in \eq{generalHyperbolicEqIntro}. The
first step is to choose the origin of the rays and the initial coefficients of
$\phi_\ell$ up to order $k+1$. Letting the ray begin at a point $z$ and expanding
$\Phi(y)$ in a Taylor series about this point,
\begin{align*}
 \Phi(y)   &= \sum_{|\beta|=0}^{k+1} \frac1{\beta!}\Phi_\beta(z)(y-z)^\beta + \text{error} \ ,
\end{align*}
we initialize the ray and Gaussian beam phase associated with $\tau_{\ell}$ as
\begin{align*}
 x_{\ell}(0;z) &= z \ , &p_{\ell}(0;z) &= \nabla_y \Phi(z) \ , \\
 M_{\ell}(0;z) &= \partial_y^2\Phi(z) + i \ {\rm Id}_{n\times n} \ , & 
 \phi_{\ell,\beta}(0;z) &= \Phi_\beta(z) \quad |\beta|=0\ , \ \ |\beta|= 3,\ldots,k+1 \ .
\end{align*}
Determining the initial coefficients for the amplitudes involves a more
complicated procedure. As with the phase, we expand all of the amplitude
functions in Taylor series,
\begin{align*}
 A_{{\ell},j}(y)    &= \sum_{|\beta|=0}^{k-2j-1} \frac1{\beta!}A_{{\ell},j,\beta}(z)(y-z)^\beta + \text{error} \ .
\end{align*}
Next, we look at the time derivatives of $u_k$ at $t=0$, to the lowest order in
$\varepsilon$. We equate the coefficients of $(y-z)$ to the the corresponding
terms in the Taylor expansions of $B_{{\ell},j}$ and recalling that
\begin{align*}
\partial_t\phi_{\ell}(0,0;z)=\dot{\phi}_{\ell,0}(t;z)-\dot{x}_\ell(0;z)\cdot p_\ell(0;z)=-H_\ell(0,{x}_\ell(0;z),p_\ell(0;z))=:\tau_{\ell}\ , 
\end{align*}
we obtain the following $m\times
m$ system of linear equations:
\begin{align*}
\begin{bmatrix}
1 &\cdots& 1 \\
\vdots& & \vdots \\
(i\tau_0)^{\ell} &\cdots& (i\tau_{m-1})^{\ell}\\
\vdots& & \vdots \\
(i\tau_0)^{m-1} &\cdots& (i\tau_{m-1})^{m-1}
\end{bmatrix}
\begin{bmatrix}
a_{0,0,\beta}\\
\vdots \\
a_{{\ell},0,\beta} \\
\vdots\\
a_{m-1,0,\beta}
\end{bmatrix}
=
\begin{bmatrix}
A_{0,0,\beta}\\
\vdots \\
A_{{\ell},0,\beta} \\
\vdots\\
A_{m-1,0,\beta}
\end{bmatrix} \ .
\end{align*}
Since the $\tau_{\ell}$ are distinct, this Vandermonde matrix is invertible, so
the solution will give the initial coefficients for the first amplitude for each
of the $m$ Gaussian beams. If we proceed with the next orders in $\varepsilon$,
we would obtain the same $m\times m$ linear system for $[a_{0,j,\beta}, \ldots,
a_{m-1,j,\beta}]^{\sf T}$, except that the right hand side will not only depend
on the Taylor coefficients $B_{{\ell},j,\beta}$, but also on previously computed
$a_{{\ell},q,\gamma}$, $q<j$, coefficients and their time derivatives. Thus, all
of the necessary initial coefficients for each of the $m$ Gaussian beams can be
computed sequentially. Summarizing this construction, we have that at $t=0$ and
${\ell}=0,\ldots,m-1$,
\begin{align}\lbeq{hyperbolicInitDataGB}
\partial_t^{\ell} u_k(0,y) = \left(\frac{1}{2\pi\varepsilon}\right)^\frac{n}{2} \int_{K_0}\varepsilon^{-{\ell}}\sum_{j=0}^N \varepsilon^j \rho_\eta(y-z) b_{{\ell},j}(y)e^{i\phi(y)/\varepsilon} dz + \Ordo(\varepsilon^\infty)\ ,
\end{align}
where $\phi(y)$ is the Taylor expansion of $\Phi(y)+i|y-z|^2/2$ to order $k+1$
and each $b_{{\ell},j}$ is the same as the Taylor expansion of $B_{{\ell},j}$ up
to order $k-2j-1$. The $\Ordo(\varepsilon^\infty)$ term is present because some
of the time derivatives fall on the cutoff function $\rho_\eta(y-x_\ell(t;z))$.
The contributions of such terms decays exponentially as $\varepsilon\to 0$,
since the derivatives of $\rho_\eta(y-x_\ell(t;z))$ are compactly supported and
vanish near $y=z$.

\begin{remark}
For ease of notation and exposition, in \eq{generalHyperbolicEqIntro} we
have taken the phase $\Phi$ to be the same for all of the $m$ initial data,
however, this is not a requirement. We can form $m$ different Gaussian beam
superpositions that each satisfy one of these $m$ conditions with a specific
phase and the rest with zero initial data. Then, summing these $m$ solutions we
obtain a more general solution of \eq{generalHyperbolicEqIntro} with $m$ different
phase functions for each of initial data piece.
\end{remark}

\begin{remark}\lbrem{imagParamM}
In the initialization of $M_\ell(0;z)$ we take its imaginary part to
be given by $i\,{\rm Id}_{n\times n}$ for simplicity. All of the results in this
paper can be carried out if we instead took the imaginary part to be given by
$i\gamma\,{\rm Id}_{n\times n}$, for some constant $\gamma>0$, and adjusted the
normalization constant in \eq{superpositionForHyperbolic} appropriately.
However, it is important to note that the constants throughout this paper will
depend on $\gamma$ and that in general, we expect that increasing $\gamma$ will
increase the evolution error in the Gaussian beams and that decreasing $\gamma$
will increase the error in approximating the initial data. 
\end{remark}

This completes the construction of the Gaussian beam superpositions $u_k$ for
the initial value problem \eq{generalHyperbolicEqIntro} for general $m$-th order
strictly hyperbolic operators. From this point, we will assume that the
parameter $\eta$ is chosen as $\eta=\infty$ for the first order superposition
$u_1$ and that for higher order superpositions it is taken small enough (and
independent of $z$ and $\ell$) to make
$\Im\phi_\ell(t,y-x_\ell(t;z);z)>\delta|y-x_\ell(t;z)|^2$ for $t\in[0,T]$ and
$y\in\supp\{\rho_\eta(y-x_\ell(t;z))\}$. \lem{GBphaseOKforQ} below shows that
this is always possible.


\subsection{The Schr\"odinger Equation}\lbsec{SchrodingerEq}
The construction of Gaussian beams for hyperbolic equations can be extended to
the Schr\"odinger equation by replacing the operator $P$ with a semiclassical
operator $P^\varepsilon$. Then, we can similarly construct asymptotic solutions
to $P^\varepsilon u=0$ as $\varepsilon \to 0$. In this section, we briefly
review the construction presented in \cite{LiuRalston:10} for the Schr\"odinger
equation \eq{SchrodingerEquationIntro} with a smooth external potential $V(y)$.
Note that the small parameter $\varepsilon$ represents the fast space and time
scale introduced in the equation, as well as the typical wavelength of
oscillations of the initial data.

We recall that the $k$-th order Gaussian beam solutions are of the form 
\begin{align*}
    v_k(t,y;z) = \sum_{j=0}^{\lceil k/2 \rceil -1} \varepsilon^j  \rho_\eta(y-x(t;z)) a_j(t,y-x(t;z);z)e^{i\phi(t,y-x(t;z);z)/\varepsilon},
\end{align*}
where the phase and amplitudes are given in \eq{PhaseWithZ} and \eq{AmpWithZ}
and $\rho_\eta$ is the cutoff function \eq{cutoff}. 
Furthermore, the subindex ``$\ell$'' has been suppressed since for the
Schr\"odinger equation there is only one choice for $H(t,x,p)$, namely
\be{HSsystem}
  H(t,x,p) = \frac{|p|^2}{2}+V(x)\ .
\ee
The system \eq{Hsystem} for the bicharacteristics $(x(t;z), p(t;z))$ is
then given by
\begin{align*}
\dot x & = p \ , &  x(0;z)&=z \ , \\
\dot p &=-\nabla_y V, & p(0;z)&=\nabla_y \Phi(z) \ .
\end{align*}
The equations for the phase and amplitude Taylor coefficients are derived in the
same way as for the strictly hyperbolic equations. The phase coefficients along
the bicharacteristic curve satisfy
\begin{align*}
\dot \phi_0 &= \frac{|p|^2}{2} - V(x(t))\ ,  \\
\dot M &= -M^2 -  \partial_y^2 V(x(t)) \ , \\
\dot \phi_\beta &= - \sum_{|\gamma |=2}^{|\beta|} \frac{(\beta-1)!}{(\gamma-1)!(\beta-\gamma)!} \phi_\gamma \phi_{\beta-\gamma+2} - \partial^\beta_y V, \quad |\beta| = 3,\ldots,k+1 \ .
\end{align*}
The amplitude coefficients are obtained by recursively solving transport equations for $a_{j, \beta}$ with $|\beta|\leq k-2j-1$, starting from 
$$
\dot a_{0, 0}=-\frac{1}{2}a_{0, 0}\mbox{Tr}(M(t; z)).
$$
These equations when equipped with the following initial data 
\begin{align*}
\phi_0(0;z)       &=\Phi_0(z)\ , &
M(0;z)            &=\partial_y^2\Phi(z)+i\ {\rm Id}_{n\times n}\ , \\
\phi_\beta(0;z)   &= \Phi_\beta(z), \quad |\beta|=3,\ldots,k+1 \ , &
a_{j, \beta}(0;z) &= A_{j, \beta}(z) \ ,
\end{align*}
have global in time solution, thus $v_k(t,y; z)$ is well-defined for all $0\leq t \leq T$. 

The $k$-th order Gaussian beam superposition is finally formed as
\begin{align*}
    u_k(t,y)=\left(\frac{1}{2\pi\varepsilon}\right)^{\frac{n}{2}} \int_{K_0} v_k(t,y;z)dz \ .
\end{align*}
As in the case of strictly hyperbolic PDEs, we will assume that the cutoff
parameter $\eta$ is chosen as $\eta=\infty$ for the first order superposition
$u_1$ and that for higher order superpositions it is taken small enough (and
independent of $z$) to make $\Im\phi(t,y-x(t;z);z)>\delta|y-x(t;z)|^2$ for
$t\in[0,T]$ and $y\in\supp\{\rho_\eta(y-x(t;z))\}$. Again, \lem{GBphaseOKforQ}
ensures that this can be done. Furthermore, we note that \rem{imagParamM}
concerning the initial choice of the imaginary part of $M(0;z)$ also applies to
the superposition for the Schr\"odinger equation.

\section{Error Estimates for Gaussian Beams}\lbsec{ErrorEstForGB}

In this section we prove the asymptotic convergence results for superpositions
of Gaussian beams given in our main result, \theo{ErrorEstWaveSchrod}. The
corner stone of our error estimates are the well-posedness estimates for each
PDE. Since they are crucial to our analysis we summarize them here.
\begin{theorem}\lbtheo{wellPosednessAll} The generic well-posedness estimate
\begin{align}\lbeq{WellPosedAll}
    \|u(t,\cdot)\|_{S} \leq \|u(0,\cdot)\|_{S} + C\varepsilon^q \int_{0}^{t} \|\Theta[u](\tau, \cdot)\|_{L^2}d\tau \ ,
\end{align}
applies to 
\begin{itemize}
\item the $m$-th order strictly hyperbolic PDE \eq{generalHyperbolicEqIntro} with $\Theta=P$, $q=0$ and $\|\cdot\|_S$ the Sobolev space-time norm,
\begin{align*}
\sum_{\ell=0}^{m-1}\left\|\partial^{\ell}_t u(t,\cdot)\right\|_{H^{m-\ell-1}}\ ,
\end{align*} 
where $H^s$ is the Sobolev $s$-norm ($H^0=L^2$),
\item the wave equation \eq{WaveEquationIntro} with $\Theta=\partial^2_t-c(y)^2\Delta$, $q=1$ and $\|\cdot\|_S$ the $\varepsilon$-scaled energy norm,
\begin{align}\lbeq{ENorm}
   ||u(t,\cdot)||_E := \left(\frac{\varepsilon^2}{2}\int_{\Real^n} \frac{|u_t|^2}{c(y)^2}+ |\nabla u|^2 dy\right)^{1/2}\ ,
\end{align}
\item and the Schr\"odinger equation \eq{SchrodingerEquationIntro} with $\Theta=P^\varepsilon$, $q=-1$ and $\|\cdot\|_S$ the standard $L^2$ norm.
\end{itemize}
\end{theorem}
\begin{proof}
 The results for the wave and Schr\"odinger equation are standard and can be found in most books on PDEs. The result for $m$-th order equations is a bit more technical to prove and appears in Section~23.2 of \cite{HormanderIII} (Lemma~23.2.1). 
\end{proof}

\begin{remark}
Since the wave equation is a second order strictly hyperbolic PDE, we have two
distinct well-posedness estimates in terms of two different norms. Furthermore,
we note that $\|\cdot\|_E$ is only a norm over the class of functions that
tend to zero at infinity, which we are considering here.
\end{remark}

When \theo{wellPosednessAll} is applied to the difference between the Gaussian
beam superposition, $u_k$, and the true solution, $u$, for any one of the PDEs
that we are considering, we obtain the following estimate for $t\in[0, T]$,
\begin{align}\lbeq{wellPosedGBError}
    \|u_k(t,\cdot)-u(t,\cdot)\|_{S} \leq \|u_k(0,\cdot) - u(0,\cdot)\|_{S} + C\varepsilon^q \int_{0}^{t} \|\Theta[u_k](\tau, \cdot)\|_{L^2}d\tau \ ,
\end{align}
with the appropriate choices for $\Theta$, $q$ and $\|\cdot\|_S$. We will refer
to the first term on the right hand side as the error in approximating the
initial data or the initial data error and to the second term as the evolution
error.

Using the ideas in \cite{Tanushev:08}, we prove \theo{InitialData}, which
shows the convergence rate in $\varepsilon$ of the Gaussian beam superposition
to the initial data for any given Sobolev norm. Thus, this theorem extends a
result of \cite{Tanushev:08}, so that we can use it to estimate the initial data
error in the more general well-posedness estimates above. 

The evolution error has been estimated in the work of previous authors
\cite{LiuRalston:09,LiuRalston:10,BougachaEtal:09,RousseSwart:09}. The necessary
steps used by those authors are quite general and can be applied to any strictly
hyperbolic equation as well as any linear dispersive wave equation as long as it
is semi-classically rescaled. Following these ideas, we show in
\lem{PuInTermsOfQ} that for all of the PDEs that we consider, $\Theta[u_k]$ can
be written in the form
\begin{align}\lbeq{PinTermsOfQ}
  \Theta[u_k] = \varepsilon^{k/2-q}\sum_{j=1}^J \varepsilon^{r_j} (\Op_{\alpha_j,g_j,\eta} f_j)(t,y)
  +\Ordo(\varepsilon^\infty) \ ,\qquad
  r_j\geq 0,\quad f_j\in L^2(K_0),
\end{align}
so that the key to estimating the evolution error is
precise norm estimates in terms of $\varepsilon$ of the oscillatory integral
operators $\Op_{\alpha_j,g_j,\eta}:L^2(K_0)\mapsto
L^\infty([0,T]; L^2(\Real^n))$, defined as follows. For a fixed $t\in[0,T]$, a
multi-index $\alpha$, a compact set $K_0\subset\Real^n$, a cutoff function
$\rho_\eta$ \eq{cutoff} with cutoff radius $0<\eta\leq\infty$ and a function
$g(t,y;z)$, we let
\begin{align}\lbeq{Operator}
   & (\Op_{\alpha,g,\eta} w)(t,y) \\ & \qquad := \varepsilon^{-\frac{n+|\alpha|}{2}}\int_{K_0} w(z)g(t,y;z)(y-x(t;z))^\alpha e^{i\phi(t,y-x(t;z);z)/\varepsilon} \rho_\eta(y-x(t;z))dz \ , \notag
\end{align}
with the functions $g(t,y;z)$, $\phi(t,y;z)$ and $x(t;z)$ satisfying for all $t\in[0,T]$,
\begin{enumerate}
\renewcommand{\theenumi}{(A\arabic{enumi})}
\renewcommand{\labelenumi}{(A\arabic{enumi})}
\item $x(t;z)\in C^\infty([0,T]\times K_0)$,
\item $\phi(t,y;z), g(t,y;z) \in C^\infty([0,T]\times \Real^n\times K_0)$, 
\item $\nabla\phi(t,0;z)$ is real and there is a constant $C$ such that for all $z,z'\in K_0$,
\begin{align*}
  |\nabla_y\phi(t,0;z)-\nabla_y\phi(t,0;z')|+|x(t;z)-x(t;z')|\geq C|z-z'|\ ,
\end{align*}
\item for $|y|\leq 2\eta$ (or for all $y$ if $\eta=\infty$), there exists a constant $\delta$ such that for all $z\in K_0$,
\begin{align*}
  \Im\phi(t,y;z)\geq \delta|y|^2\ ,
\end{align*}
\item for any multi-index $\beta$, there exists a constant $C_\beta$, such that
\begin{align*}
   \mathop{\sup_{z\in K_0}}_{y\in\Real^n} \left|\partial_y^\beta g(t,y;z)\right|\leq C_\beta\ .
\end{align*}
\end{enumerate}
With this definition, the following norm estimate of $\Op_{\alpha,g,\eta}$ will be proved in \sect{NormEstQ}.
\begin{theorem}\lbtheo{QL2norm}
Under the assumptions (A1)--(A5),
\begin{align*}
   \sup_{t\in[0,T]}\left\|\Op_{\alpha,g,\eta}\right\|_{L^2} \leq C(T).
\end{align*}
\end{theorem}
This theorem improves the norm estimate of $\Op_{\alpha,g,\eta}$ given in
\cite{LiuRalston:09,LiuRalston:10}, which has an additional factor
$\varepsilon^{-\gamma}$ in the right hand side, with $\gamma=(n-1)/4$ for the
wave equation and $\gamma=n/4$ for the Schr\"odinger equation. Instead of
estimating the integral directly, we follow the arguments in
\cite{BougachaEtal:09,RousseSwart:07} to relate the estimate of the oscillatory
integral to the operator norm, through the use of an adjoint operator. An
essential ingredient in estimating the operator norm is the non-squeezing lemma
(\lem{nonsqueezing}), which states that the distance between two physical points
is comparable to the distance between their Hamiltonian trajectories measured in
phase space, even in the presence of caustics. Using \theo{QL2norm} we are able
to prove the same convergence rate for Gaussian beam superpositions over
physical space that is achieved in \cite{BougachaEtal:09} for beam superposition
carried out in full phase space. Thus, we improve on the error estimates given in
\cite{LiuRalston:09,LiuRalston:10} to obtain error estimates for the Gaussian
beam superposition that are independent of dimension as given in
\theo{ErrorEstWaveSchrod}.

We conclude this section with two remarks.

\begin{remark}
The assumption of $C^\infty$ smoothness for all functions is made for simplicity
to avoid a too technical discussion about precise regularity requirements. In
this sense, \theo{ErrorEstWaveSchrod} and \theo{QL2norm} can be sharpened, since
they will be true also for less regular functions.
\end{remark}

\begin{remark}
If the condition in assumption (A4) is satisfied for all $y$ there is no need
for the cutoff function in the definition of the operator $\Op$ in
\eq{Operator}. We treat this case by taking $\eta=\infty$ and defining
$\rho_\infty\equiv 1$. The operators with $\eta=\infty$ are used in the case of
first order Gaussian beams.
\end{remark}

\subsection{Gaussian Beam Phase}

In this section, we show that the Gaussian beam phase $\phi$ given in
\eq{PhaseWithZ} is an admissible phase for the operators $\Op_{\alpha,g,\eta}$.
We begin with a lemma based on the regularity of the Hamiltonian flow map,
$S_t$, stating that the difference $|z-z'|$ is comparable to the {\em sum}
$|p(t;z)-p(t;z')| +|x(t;z)-x(t;z')|$. Note, however, that because of caustics
it is not true that
$|z-z'|$ is related in this way to either of the individual terms
$|p(t;z)-p(t;z')|$ or $|x(t;z)-x(t;z')|$. 
\begin{lemma}[Non-squeezing lemma]\lblem{nonsqueezing} 
Let $S_t$ be a Hamiltonian flow map, $$(x(t;z),p(t;z))=S_t(x(0;z),p(0;z)),$$
associated to a strictly hyperbolic PDO \eq{Hsystem} or to the Schr\"odinger
operator \eq{HSsystem}. Let $K_0$ be a compact subset of $\Real^n$ and assume
that $p(0;z)$ is Lipschitz continuous in $z \in K_0$ for the flow associated
with the Schr\"odinger operator. Additionally, assume that $\inf_{z\in
K_0}|p(0;z)|= \delta>0$ for the flow associated with the strictly hyperbolic
PDO. Under these conditions, there exist positive constants $c_1$ and $c_2$
depending on $T$ and $\delta$, such that
\begin{align}\lbeq{NonSqueezeIneq}
   c_1 |z-z'| \leq |p(t;z)-p(t;z')|+|x(t;z)-x(t;z')|\leq c_2 |z-z'| \ ,
\end{align}
for all $z,z'\in K_0$ and $t\in[0,T]$. 
\end{lemma}

\begin{proof} 
We prove the result for the flows associated with the two types of operators
separately, however, we use the common notation,
\begin{align*}
 Z&=(z,p_0)=(z,p(0;z))\ , &Z'&=(z',p_0')=(z',p(0;z'))\ , \\
X&=S_t(Z)=(x,p)=(x(t;z),p(t;z))\ , &X'&=S_t(Z')=(x',p')=(x(t;z'),p(t;z')) \ .
\end{align*}
We begin with the flow associated with the Hamiltonian for the Schr\"odinger operator. Let us introduce the set $\Kcal_0 = \{ (z,p(0;z)):z\in K_0\}$ and note that $S_t$ is invertible with inverse $S_{-t}$ and regular for all $t$ so that
\begin{align*}
   \sup_{t\in[0,T]}\sup_{\tilde{Z}\in \conv({\Kcal}_0)} \left\|\frac{\partial S_t(\tilde{Z})}{\partial Z}\right\|\leq C\ , \qquad\qquad
  \sup_{t\in[0,T]}\sup_{\tilde{X}\in \conv(S_t({\Kcal}_0))} \left\|\frac{\partial S_{-t}(\tilde{X})}{\partial X}\right\|\leq C\ ,
\end{align*}
 where $\conv(E)$ denotes the convex hull of the set $E$. Now, noting that
\begin{equation}\lbeq{FlowBound}
   X-X' = \int_0^1\frac{d}{ds} S_t(sZ+(1-s)Z')ds = 
   \int_0^1\frac{\partial S_t(sZ+(1-s)Z')}{\partial Z}(Z-Z')ds \ , 
\end{equation}
and taking $\ell_1$ norms, since $sZ+(1-s)Z'\in \conv(K_0)$,
we have
$$
   |x-x'|+|p-p'|=||X-X'||_1 \leq C||Z-Z'||_1= C(|z-z'|+|p_0-p_0'|)\leq C' |z-z'|\ ,
$$
where we have used the Lipschitz continuity of $p_0$ in $z$. This gives the right half of \eq{NonSqueezeIneq}. By the equivalent of \eq{FlowBound} for $S_{-t}$, we have
\begin{align*}
  |z-z'|\leq |z-z'|+|p_0-p_0'| = \|Z-Z'\|_1 \leq C\|X-X'\|_1=C(|x-x'|+|p-p'|) \ ,
\end{align*}
which completes the proof of the lemma for the Hamiltonian flow associated with the Schr\"odinger operator.

For the flow associated with the strictly hyperbolic operator, we follow the same idea, but we have to be careful near $|p|=0$. Thus, in addition to $\Kcal_0$, we introduce the sets
 \begin{align*}
 B_\delta &= \left\{ (z,p)\ :\ z\in K_0,\ |p| < \delta\right\},\\
 \tilde{\Kcal}_0 &= {\rm conv}\left(\Kcal_0 \cup B_\delta \right) \ .
 \end{align*}
 Note that $S_t$ is regular away from $|p_0|=0$ by \lem{HyperbolicHamiltonianLemma} as $|p(t)|\neq 0$,  for all $t$ so
\begin{align*}
   \sup_{t\in[0,T]}\sup_{\tilde{Z}\in \tilde{\Kcal}_0\setminus B_{\delta/2}} \left\|\frac{\partial S_t(\tilde{Z})}{\partial Z}\right\|\leq C.
\end{align*}
Thus, again by \eq{FlowBound} we obtain,
\be{xpbound}
   |x-x'|+|p-p'|=||X-X'||_1 \leq C||Z-Z'||_1= C(|z-z'|+|p_0-p_0'|)\leq C'|z-z'| \ ,
\ee
provided that $\tilde{p}(s)=(1-s)p_0+sp_0'$ satisfies $\inf_{0\leq s\leq
1}|\tilde{p}(s)|\geq \delta/2$, which guarantees that
$sZ+(1-s)Z'\in\tilde{\Kcal}_0\setminus B_{\delta/2}$ for $0\leq s\leq 1$. 
On the other hand, suppose that $\inf_{0\leq s\leq 1} |\tilde p(s)|< \delta/2$.
We define $p^*$ as the point on the line connecting $p_0$ to $p_0'$ with
smallest norm and let $s^* = \argmin_{0 \leq s \leq 1}|\tilde p(s)|$ so that 
$p^*=\tilde{p}(s^*)$. Then
\begin{align*}
\frac{\delta}{2}>|p^*|=|p_0-s^*(p_0-p_0')| \geq |p_0|-s^*|p_0-p_0'|\geq \delta -|p_0-p_0'|.
\end{align*}
Hence, $|p_0-p_0'| \geq \delta/2$. Now, let
\begin{align*}
d = \sup_{ t \in [0, T]}\sup_{Z\in \Kcal_0}||S_t(Z)||_1 < \infty,
\end{align*}
so that  
\begin{align*}
\|X-X'\|_1 & \leq \|X\|_1+\|X'\|_1\leq 2d \leq  2d \frac{2}{\delta}|p_0-p_0'| \leq C\|Z-Z'\|_1,
\end{align*}
which shows that \eq{xpbound} holds for all $Z,\ Z'\in\Kcal_0$. This gives the
right half of \eq{NonSqueezeIneq}. For the left half of \eq{NonSqueezeIneq}, we note that by
\lem{HyperbolicHamiltonianLemma}, $\inf_{t\in[0,T],z\in K_0} |p(t;z)|\geq
\tilde{\delta}>0$, so that we can consider the inverse map $S_{-t}$ in exactly the
same way as $S_t$ above and show that $||Z-Z'||_1\leq C||X-X'||_1$. Thus we
obtain the left half of \eq{NonSqueezeIneq},
\begin{align*}
   |z-z'|\leq |z-z'|+|p_0-p_0'|= ||Z-Z'||_1 \leq C||X-X'||_1\leq C(|x-x'|+|p-p'|)\ .
\end{align*}
Thus, the proof of the lemma is complete.
\end{proof}

We are now ready to show that the phase function for a $k$-th order Gaussian beam is an admissible phase for the operators $\Op_{\alpha,g,\eta}$.

\begin{lemma}\lblem{GBphaseOKforQ}
Let $\phi(t,y;z)$ and $x(t;z)$ be the phase and central ray associated to a
$k$-th order Gaussian beam for \eq{generalHyperbolicEqIntro},
\eq{WaveEquationIntro} or \eq{SchrodingerEquationIntro} as given in
\sect{SuperpositionsOfGB} and \sect{SchrodingerEq}. The rays $x(t;z)$ satisfy
(A1) and $\phi(t,y;z)$ satisfies assumptions (A2) through (A4) if $\eta$ is
sufficiently small. In the case of $k=1$, $\eta$ can take any value in
$(0,\infty]$.
\end{lemma}
\begin{proof}
The smoothness assumptions (A1) and (A2) follow from smoothness of initial data and smoothness of the coefficients in the underlying PDE. By definition, $\nabla_y\phi(t,0;z) = p(t;z)$ and (A3) follows from the non-squeezing lemma, \lem{nonsqueezing}. Finally, since the lower order terms of $\phi$ are real,
\begin{align*}
  \Im\phi(t,y;z) = y\cdot (\Im M(t;z))y + 
  \sum_{|\beta|=3}^{k+1} \frac{1}{\beta!}\Im\phi_\beta(t;z)y^\beta \ .
\end{align*}
Recalling that by \lem{HessianLemma}, $\Im M(t;z)$ is positive definite, we therefore have that when
$|y|\leq 2\eta$,
\begin{align*}
  \Im\phi(t,y;z) &\geq C |y|^2 - \sum_{|\beta|=3}^{k+1} \frac{1}{\beta!}\|\Im\phi_\beta(t;z)\|_{L^\infty} |y|^{|\beta|} \\
  &\geq C |y|^2 - |y|^2\sum_{|\beta|=3}^{k+1} \frac{1}{\beta!}\|\Im\phi_\beta(t;z)\|_{L^\infty} (2\eta)^{|\beta|-2} \\
  &\geq |y|^2\left(C-2\eta \sum_{|\beta|=3}^{k+1} \frac{1}{\beta!}\|\Im\phi_\beta(t;z)\|_{L^\infty} (2\eta)^{|\beta|-3}\right)\\
  &\geq \delta(\eta)|y|^2,
\end{align*}
where the constant $\delta(\eta)$ is positive for small enough $\eta$ and
independent of $t\in[0,T]$, since $\phi_\beta(t;z)$ are smooth functions of $t$.
This shows that $\phi(t,y;z)$ satisfies (A4). When $k=1$, there are only
quadratic terms in the phase and in this case $\phi(t,y;z)$ will satisfy (A4)
for any choice of $\eta\in(0,\infty]$.
\end{proof}

\subsection{Representations of $P[u_k]$ in Terms of $\Op_{\alpha,g,\eta}$}

In this section, we show that several of the intermediate quantities in the proof of \theo{ErrorEstWaveSchrod} can be written as sums involving the operators $\Op_{\alpha,g,\eta}$.

\begin{lemma}
\lblem{PuInTermsOfQ} 
Let $P$ be an $m$-th order strictly hyperbolic operator and $P^\varepsilon$ a
semiclassical Schr\"odinger operator, satsifying the assumptions stated for
\eq{generalHyperbolicEqIntro} (in \sect{GBforHyperbolicEq}) and
\eq{SchrodingerEquationIntro} respectively. Let $u_k$ be the corresponding
Gaussian beam superpositions given in \sect{SuperpositionsOfGB} and
\sect{SchrodingerEq}. Then $P[u_k]$ and $P^\varepsilon[u_k]$ can be expressed as
a finite sum of the operators $\Op$:
\begin{align*}
\left. \begin{array}{l}P[u_k]\\ \\P^\varepsilon[u_k]\end{array} \right\}= \left(\frac{1}{2\pi}\right)^\frac{n}{2} \sum_{j=1}^{J} \varepsilon^{\ell_j}\left( \Op_{\alpha_j,g_j,\eta}\  \chi_{K_0}\right)(t,y)  +\Ordo(\varepsilon^\infty), \ \left\{ \begin{array}{l} \ell_j \geq k/2-m+1 \\ \\ \ell_j \geq k/2+1  \end{array} \right.	,
\end{align*}
with $\chi_{K_0}(z)$ the characteristic function on $K_0$ and
$\Op_{\alpha_j,g_j,\eta}$ satisfying assumptions (A1)-(A5) if $\eta$ is
sufficiently small. In the case of $k=1$, $\eta$ can take any value in
$(0,\infty]$.
\end{lemma}

\begin{proof}
For a single Gaussian beam $v_{k,\ell}(t,y;z)$ for the hyperbolic operator,
following the discussion in \cite{Ralston:82} and \sect{GBforHyperbolicEq}, we have
\begin{align*}
  P[v_{k,\ell}(t,y;z)] = \sum_{r=-m}^{\lceil k/2\rceil-1} \varepsilon^r \rho_\eta(y-x_{\ell}(t;z)) c_{r,\ell}(t,y;z) e^{i \phi(t,y-x_{\ell}(t;z);z)/\varepsilon} + E_k(t,y;z)\ ,
\end{align*}
where $E_k$ contains terms that are multiplied by derivatives of the cutoff
function. For first order beams $\eta=\infty$ and $\rho_\infty\equiv 1$ making
$E_1(t,y;z)\equiv 0$. For higher order beams, $E_k\equiv 0$ in an $\eta$
neighborhood of $x_{\ell}(t;z)$ and since $\eta$ is small enough so that the
imaginary part of $\phi_{\ell}$ is strictly positive for $\eta \leq
|y-x_{\ell}(t;z)|\leq 2\eta$, $E_k$ will decay exponentially as $\varepsilon\to
0$. Thus, $E_k = \Ordo(\varepsilon^\infty)$ for all orders of beams
and $t\in [0,T]$.

Each $c_r$ can be expressed in terms of the symbols of $P$:
\begin{align*}
 c_{-m+j,\ell}(t,y;z) &= La_{{\ell},j-1} + \sigma_m\left(t,y,\partial_t\phi_{\ell},\nabla_y\phi_{\ell}\right)a_{{\ell},j} + R_{\ell,j}(t,y;z) \ ,
\end{align*}
where $a_{{\ell},j}\equiv 0$ for $j\not\in [0,\lceil k/2 \rceil -1]$ and $L$ is given by (using the Einstein summation),
\begin{align*}
La = -i\left( \frac{\partial \sigma_m}{\partial \tilde{p}_j }(\tilde{y},\tilde{\nabla}\phi_{\ell}) \frac{\partial a}{\partial \tilde{y}_j}\right)
     -\left( \frac{i}{2}\frac{\partial^2 \sigma_m}{\partial \tilde{p}_j\tilde{p}_q}(\tilde{y},\tilde{\nabla{}}_y\phi_{\ell}) \phi_{{\ell},\tilde{y}_j\tilde{y}_q}+ \sigma_{m-1}(\tilde{y},\tilde{\nabla{}}\phi_{\ell})\right) a \ ,
\end{align*}
with $\tilde{\nabla}=(\partial_t,\nabla)$, $\tilde{y}=(t,y)$ and
$\tilde{p}=(\tau,p)$. The function $R_{\ell,0}(t,y;z)\equiv 0$ and the functions
$R_{\ell,j}(t,y;z)$ for $j>0$ are complicated functions of the $m-2$ and lower
order symbols of $P$ and the functions $\phi_{\ell}$, $a_{{\ell},0}$, \ldots,
$a_{{\ell},j-1}$ and their derivatives. We note that since $a_{{\ell},j}(t,y;z)$ are
compactly supported in $z\in K_0$, so are the functions $c_{r,\ell}(t,y;z)$.

By the construction of a Gaussian beam, $c_{r,\ell}$ vanishes up to
order ${k-2(r+m)+1}$ on $y=x_\ell(t;z)$. Note that ${k-2(r+m)+1}$ may be negative,
in which case, $c_{r,\ell}$ is not necessarily $0$ on $y=x_\ell(t;z)$. Thus, by
Taylor's remainder formula,
\begin{align*}
 c_{r,\ell}(t,y;z) = \sum_{|\alpha|=[k-2(r+m)+2]_+} c_{r,\ell,\alpha}(t,y;z)\;(y-x_{\ell}(t;z))^\alpha \ ,
\end{align*}
for some coefficient functions $c_{r,\ell,\alpha}(t,y;z)$ that are compactly
supported in $z\in K_0$ and $[a]_+=\max[a,0]$. For the superposition $u_k(t,y)$,
we have
\begin{align*}
 &P[u_k] = \left(\frac{1}{2\pi\varepsilon}\right)^\frac{n}{2} \sum_{{\ell}=0}^{m-1} \int_{K_0} P[v_{k,\ell}(t,y;z)]dz \\
        &\ = \left(\frac{1}{2\pi\varepsilon}\right)^\frac{n}{2} \sum_{{\ell}=0}^{m-1} \left(\sum_{r=-m}^{\lceil k/2\rceil-1} \int_{K_0}\varepsilon^r \rho_\eta(y-x_{\ell}(t;z)) c_{r,\ell}(t,y;z) e^{i\phi_{\ell}(t,y-x_{\ell}(t;z);z)/\varepsilon}dz\right) ,
\end{align*}
modulo additive terms that are $\Ordo(\varepsilon^\infty)$. Now substituting for
$c_{r,\ell}$ we can rewrite this expression in terms of the operators $\Op$,
\begin{align*}
 P[u_k] &= \left(\frac{1}{2\pi\varepsilon}\right)^\frac{n}{2} \sum_{{\ell}=0}^{m-1}\Bigg( \sum_{r=-m}^{\lceil k/2\rceil-1} \Bigg( \sum_{|\alpha|=[k-2(r+m)+2]_+}
 \int_{K_0}\varepsilon^r \rho_\eta(y-x_{\ell}(t;z)) \\
 &\qquad\qquad\qquad\qquad\qquad\times  c_{r,\ell,\alpha}(t,y;z)(y-x_{\ell}(t;z))^\alpha e^{i\phi_{\ell}(t,y-x_{\ell}(t;z);z)/\varepsilon}dz\Bigg)\Bigg) \\
&= \left(\frac{1}{2\pi}\right)^\frac{n}{2} \sum_{{\ell}=0}^{m-1} \Bigg(\sum_{r=-m}^{\lceil k/2\rceil-1} \Bigg( \sum_{|\alpha|=[k-2(r+m)+2]_+}\varepsilon^{r+|\alpha|/2}\left( \Op_{\alpha,c_{r,\ell,\alpha},\eta} \ \chi_{K_0}\right)(t,y) \Bigg)\Bigg)  \ ,
\end{align*}
modulo additive terms that are $\Ordo(\varepsilon^\infty)$ and where
$\chi_{K_0}(z)$ is the characteristic function on $K_0$ and the functions
$c_{r,\ell,\alpha}(t,y;z)$ satisfy condition (A5) since the coefficients of $P$,
$\phi_{\ell,\beta}$ and $a_{\ell,j,\beta}$ are all smooth. By
\lem{GBphaseOKforQ} the operators also satisfy (A1)-(A4) under the condition on
$\eta$.
 
 To simplify the notation, let $j=1,\ldots,J<\infty$ enumerate all of the combination of ${\ell}$, $r$, and $\alpha$ in the triple sum above and rewrite the sums as
\begin{align*}
P[u_k]= \left(\frac{1}{2\pi}\right)^\frac{n}{2} \sum_{j=1}^{J} \varepsilon^{\ell_j}\left( \Op_{\alpha_j,g_j,\eta}\  \chi_{K_0}\right)(t,y)  +\Ordo(\varepsilon^\infty) \ ,
\end{align*}
with $\ell_j\geq (k/2-m+1)$ and $g_j = c_{r_j,\ell_j,\alpha_j}$. Thus, we have the desired result for $P$.

Now, for each Gaussian beam $v_k(t,y;z)$ for the Schr\"odinger equation defined in \sect{SchrodingerEq}, following \cite{LiuRalston:10}, we compute
\begin{align*}
P^\varepsilon[v_k]=\rho_\eta  P^\varepsilon\left[\sum_{j=0}^{\lceil k/2 \rceil -1} \varepsilon^j  a_j  e^{i\phi/\varepsilon}\right]+E_k(t,y;z)
\end{align*}
where again $E_k=\Ordo(\varepsilon^\infty)$. Note that
\begin{align*}
e^{-i\phi/\varepsilon} P^\varepsilon[ae^{i\phi/\varepsilon}] = a G(t,y)-i\varepsilon La-\frac{\varepsilon^2}{2}\triangle a\ ,
\end{align*}
with
\begin{align*}
G =\partial_t\phi+\frac{1}{2}|\nabla_y \phi|^2 + V(y) \qquad \mbox{ and } \qquad
L =\partial_t+\nabla_y\phi \cdot \nabla_y+\frac{1}{2}\triangle_y\phi \ .
\end{align*}
Thus, we have
\begin{align*}
	P^\varepsilon\left[\sum_{j=0}^{\lceil k/2 \rceil -1} \varepsilon^j  a_j  e^{i\phi/\varepsilon}\right]= \sum_{j=0}^{\lceil k/2\rceil -1} \varepsilon^j \left[ a_jG - i\varepsilon La_j -\frac{\varepsilon}{2}\triangle_y a_j \right]e^{i\phi/\varepsilon} = \sum_{r=0}^{\lceil k/2\rceil+1} \varepsilon^r d_r e^{i\phi/\varepsilon}\ ,
\end{align*}
where for convenience $a_j\equiv 0$ for $j\not\in[0,\lceil k/2 \rceil -1]$ and
\begin{align*}
d_{r} &=a_r G-iLa_{r-1}-\frac{1}{2}\triangle_y a_{r-2}\ , \qquad r=0, \ldots, \lceil k/2 \rceil+1  \ .
\end{align*}
By construction of the phase coefficients in \sect{SchrodingerEq}, we have that
on ${y=x(t;z)}$, $G$ vanishes to order $k+1$ and by construction of the amplitude
coefficients, the quantity ${-iLa_{r-1}-\frac12\triangle_y a_{r-2}}$ vanishes to order
$k-2r+1$. Thus, $d_r$ vanishes to order $k-2r+1$ on $y=x(t;z)$, where again we remark that if $k-2r+1$ is negative, $d_r$ does not vanish on $y=x(t;z)$. By Taylor's remainder formula, we have 
\begin{align*}
d_r(t,y;z) = \sum_{|\alpha|=[k+2-2r]_+} d_{r,\alpha}(t,y;z)\;(y-x(t;z))^\alpha\ , 
\end{align*}
with $d_{r,\alpha}(t,y;z)$ compactly supported in $z \in K_0$. Hence, 
\begin{align*}
	P^\varepsilon[v_k] = \rho_\eta \sum_{r=0}^{\lceil k/2\rceil +1} \varepsilon^r \left( \sum_{|\alpha|=[k+2-2r]_+} d_{r,\alpha}(y-x(t;z))^\alpha \right) e^{i\phi/\varepsilon} + \Ordo(\varepsilon^\infty) \ .
\end{align*}
For the superposition, we have 
\begin{align*}
P^\varepsilon[u_k]&=\left(\frac{1}{2\pi\varepsilon}\right)^{\frac{n}{2}} \int_{K_0} P^\varepsilon[v_k]dz \\
&=\left(\frac{1}{2\pi}\right)^{\frac{n}{2}} \sum_{r=0}^{\lceil k/2\rceil -1} \sum_{|\alpha|=[k+2-2r]_+}\varepsilon^{r+|\alpha|/2} ( \Op_{\alpha,d_{r,\alpha},\eta} \chi_{K_0} )(t,y)  + \Ordo(\varepsilon^\infty)\ .
\end{align*}
As above, the smoothness of $V(y)$, $\phi_\beta$ and $a_{j,\beta}$
guarantees that $d_{r,\alpha}$ satisfies (A5), while
\lem{GBphaseOKforQ} ensures that also (A1)-(A4) are true for $\Op_{\alpha,d_{r,\alpha},\eta}$ under the condition on $\eta$.
 Again, we let $j=1,\ldots,J<\infty$ enumerate all of the combination of $r$ and $\alpha$ in the double sum above. Thus,
\begin{align*}
P^\varepsilon[u_k]= \left(\frac{1}{2\pi}\right)^\frac{n}{2} \sum_{j=1}^{J} \varepsilon^{\ell_j}\left( \Op_{\alpha_j,g_j,\eta}\  \chi_{K_0}\right)(t,y)  +\Ordo(\varepsilon^\infty) \ ,
\end{align*}
with $\ell_j\geq (k/2+1)$ and $g_j = d_{r_j,\alpha_j}$. Thus, the proof is complete.
\end{proof}

\subsection{Initial Data Errors}\lbsec{InitalDataErrors}

In addition to establishing the role of the $\Op$-operators and estimating
their norms, 
we need to consider the convergence of the Gaussian beam 
superposition to the initial data to be able to prove
\theo{ErrorEstWaveSchrod}. We start with the following lemma. 

\begin{lemma}\lblem{ApproxLemma} Let $\Phi\in C^\infty(K_0)$ be a
real-valued function and $A_j\in C^\infty_0(K_0)$. With $k\geq 1$, define
\begin{align*}
 u(y)     &= \sum_{j=0}^N \varepsilon^j A_j(y)e^{i\Phi(y)/\varepsilon} ,\\
 u_k(y) &= \left(\frac{1}{2\pi\varepsilon}\right)^{\frac{n}{2}}\int_{\Real^n}  \rho_\eta(y-z) \sum_{j=0}^N \varepsilon^j a_j(y-z;z) e^{i\phi(y-z;z)/\varepsilon-|y-z|^2/2\varepsilon} dz \ ,
\end{align*}
where $\phi(y-z;z)$ is the $k+1$ order Taylor series of $\Phi$ about $z$,
$a_j(y-z;z)$ is the same as the Taylor series of $A_j$ about $z$ up to order
$k-2j-1$ (but may differ in higher order terms) and $\rho_\eta$ is the cutoff function \eq{cutoff} with $0<\eta\leq\infty$. Then for some constant $C_s$,
\begin{align*}
 \left\| u_k - u \right\|_{H^s} &\leq C_s \varepsilon^{\frac{k}{2}-s} \ .
\end{align*}
\end{lemma}

\begin{proof}
The proof of this lemma is based on a discussion in \cite{Tanushev:08}. We first
assume that $\eta<\infty$. Looking at each of the terms in the sum above
separately, the estimate depends on how well
\begin{align}\lbeq{InitDataExact}
 \varepsilon^j A_j(y)e^{i\Phi(y)/\varepsilon} 
 & =  \left(\frac{1}{2\pi\varepsilon}\right)^{\frac{n}{2}}\int_{\Real^n} \varepsilon^j A_j(y)e^{i\Phi(y)/\varepsilon-|y-z|^2/2\varepsilon} dz
\end{align}
 is approximated by 
\begin{align}\lbeq{InitDataApprox}
 \left(\frac{1}{2\pi\varepsilon}\right)^{\frac{n}{2}}\int_{\Real^n}  \varepsilon^j\rho_\eta(y-z)  a_j(y-z;z)e^{i\phi(y-z;z)/\varepsilon-|y-z|^2/2\varepsilon} dz \ .
\end{align}
Since we want to estimate the difference between these two functions in a
Sobolev norm, we need to consider differences between their $\partial_y^\beta$
derivatives in the $L^2$ norm. Since derivatives of \eq{InitDataApprox} that
fall on the cutoff function $\rho_\eta(y-z)$ vanish in a neighborhood of $z=y$
and the integrand is compactly supported in $y$ and $z$, they will be
$\Ordo(\varepsilon^\infty)$ in the $L^2$ norm and do not contribute to the
estimate. Thus, when we differentiate \eq{InitDataExact} and \eq{InitDataApprox} by $\partial_y^\beta$ and pair the results from the sequence of differentiations, the terms that will contribute to the estimate will be of the form
\begin{align}\lbeq{InitDataExactTerms}
\left(\frac{1}{2\pi\varepsilon}\right)^{\frac{n}{2}}\varepsilon^{j-\ell}C_\ell^{\{\gamma_\ell\}}\partial_y^{\delta_1}[
A_j(y)]\left(\prod_{m=1}^\ell\partial_y^{\gamma_m}\left[i\Phi(y)-|y-z|^2/2\right]\right)
\end{align}
for \eq{InitDataExact} and
\begin{align}\lbeq{InitDataApproxTerms}
\left(\frac{1}{2\pi\varepsilon}\right)^{\frac{n}{2}}\varepsilon^{j-\ell}C_\ell^{\{\gamma_\ell\}}\partial_y^{\delta_1}[
a_j(y-z;z)]\left(\prod_{m=1}^\ell\partial_y^{\gamma_m}\left[i\phi(y-z;z)-|y-z|^2/2\right]\right)
\end{align}
for \eq{InitDataApprox}, where $C_\ell^{\{\gamma_\ell\}}$ are combinatorial
coefficients. These terms are integrated over $\Real^{n}$ in $z$ and summed over
the multi-indexes $\delta_1+\gamma=\beta$, the index $\ell=1\ldots|\gamma|$, and
multi-indexes $|\gamma_1|,\ldots,|\gamma_\ell|\geq 1$,
$\gamma_1+\ldots+\gamma_\ell=\gamma$. Furthermore, we have that
$\max_m[\gamma_m]\leq |\gamma|-\ell+1$. These formulas can be obtained through
long but straightforward calculations. The important part is to recall that
$i\phi(y-z;z)-|y-z|^2/2$ is the $k+1$ order Taylor series of
$i\Phi(y)-|y-z|^2/2$ about $z$ and $a_j(y-z;z)$ agrees with the Taylor series of
$A_j(y)$ about $z$ up to order $k-2j-1$, so that \eq{InitDataApproxTerms} will
agree with the Taylor series of \eq{InitDataExactTerms} up to order
\begin{align*}
\min\left[k-2j-1-|\delta_1|,k+1-\max_{1\leq m\leq \ell}|\gamma_m|\right]&\geq
k-2j-1-|\delta_1|-|\gamma|+\ell\\
&=k-2j-1-|\beta|+\ell \ . 
\end{align*}
Thus, we have that $\|\partial_y^\beta (u_k-u)\|_{L^2}$ can be estimated by a
sum of terms of the form
\begin{align}\lbeq{InitDataTerms}
\varepsilon^{j-\ell}\Bigg\| \left(\frac{1}{2\pi\varepsilon}\right)^{\frac{n}{2}}\int_{\Real^n} &\Big[ B_\ell(y)e^{i\Phi(y)/\varepsilon-|y-z|^2/2\varepsilon}  \notag \\ 
&-\rho_\eta(y-z) b_\ell(y-z;z) e^{i\phi(y-z;z)/\varepsilon-|y-z|^2/2\varepsilon}\Big] dz  \Bigg\|_{L^2} \ ,
\end{align}
where $\ell\leq|\beta|$, $|B_\ell - b_\ell| =
\Ordo(|y-z|^{k-2j-|\beta|+\ell})$ and $|\Phi-\phi| = \Ordo(|y-z|^{k+2})$. Now, the
proof of Theorem~2.1 in \cite{Tanushev:08} can be applied directly to \eq{InitDataTerms} to obtain the
estimate,
\begin{align*}
\|\partial_y^\beta (u_k-u)\|_{L^2} \leq \sum_{j=0}^N \varepsilon^{j-\ell}C\varepsilon^{\frac{k}{2}-j-\frac{|\beta|}{2}+\frac{\ell}{2}} \leq C\varepsilon^{\frac{k}{2}-|\beta|} \ .
\end{align*}
Thus, we have the result for $\eta<\infty$. The extension for $\rho_\infty\equiv
1$ follows directly, since the cutoff $\rho_\nu$ for $\nu<\infty$ introduces
$\Ordo(\varepsilon^\infty)$ errors in the $L^2$ norm:
\begin{align*}
\left\| \left(\frac{1}{2\pi\varepsilon}\right)^{\frac{n}{2}}\int_{\Real^n} \left[ 1-\rho_\eta(y-z)\right] b_\ell(y-z;z) e^{i\phi(y-z;z)/\varepsilon-|y-z|^2/2\varepsilon} dz  \right\|_{L^2} = \Ordo(\varepsilon^\infty) \ ,
\end{align*}
as $1-\rho_\nu$ vanishes in a neighborhood $z=y$ and the integrand is compactly supported in $z$. 
\end{proof}

Using \lem{ApproxLemma}, we can estimate the asymptotic convergence rate of the
superposition solution to the initial data.
\begin{theorem}\lbtheo{InitialData}
For the Gaussian beam superposition, $u_k$, given in \sect{SuperpositionsOfGB}
and the solution, $u$, to the strictly hyperbolic PDE \eq{generalHyperbolicEqIntro},
we have
\begin{align*}
 \left\| \partial_t^{\ell} u_k(0,\cdot) -  \partial_t^{\ell} u(0,\cdot) \right\|_{H^s} &\leq C_{{\ell},s} \varepsilon^{\frac{k}{2}-{\ell}-s} \ ,
\end{align*}
for some constant $C_{{\ell},s}$ and $0\leq {\ell} \leq m-1$.

\indent Similarly, for the Gaussian beam superposition, $u_k$, given in
\sect{SuperpositionsOfGB} and the solution, $u$, to the wave equation \eq{WaveEquationIntro}, we
have
\begin{align*}
 \left\| u_k(0,\cdot) - u(0,\cdot) \right\|_{E} &\leq C \varepsilon^{\frac{k}{2}} \ .
\end{align*}

\noindent Furthermore, for the superposition, $u_k$, given in
\sect{SchrodingerEq} and the solution, $u$, to the Schr\"odinger equation
\eq{SchrodingerEquationIntro}, we have
\begin{align*}
  \left\|  u_k(0,\cdot) - u(0,\cdot)  \right\|_{L^2} &\leq C \varepsilon^{\frac{k}{2}} \ .
\end{align*}
\end{theorem}

\begin{proof}
The proof of this theorem for hyperbolic PDEs follows directly from
\lem{ApproxLemma}, since for each power of $\varepsilon$, $\partial_t^{\ell}
u_k(0,y)$ and $\partial_t^{\ell} u(0,y)$ given in \eq{hyperbolicInitDataGB} and
\eq{generalHyperbolicEqIntro}, respectively, are exactly in the assumed form in the
\lem{ApproxLemma}.

The result for the wave equation follows after noting that
\begin{align*}
\|u_k(0,\cdot)-u(0,\cdot)\|_E & \leq C\varepsilon \left(\|u_k(0,\cdot)-u(0,\cdot)\|_{H^1}+ \|\partial_t u_k(0,\cdot) -\partial_t u(0,\cdot)\|_{L^2} \right)\ .
\end{align*}
Similarly, the result for the Schr\"odinger equation follows directly from the definition of the $u_k$ and $u$ at $t=0$ in \sect{SchrodingerEq}.
\end{proof}

\subsection{Proof of \theo{ErrorEstWaveSchrod}}

We prove the results for each type of PDE separately. For the strictly
hyperbolic $m$-th order PDE \eq{generalHyperbolicEqIntro}, applying the
well-posedness estimate given in \theo{wellPosednessAll} to the difference
between the true solution $u$ and the $k$-th order Gaussian beam superposition,
$u_k$, defined in \sect{SuperpositionsOfGB}, we obtain for $t\in[0,T]$,
\begin{align*}
&\sum_{\ell=0}^{m-1}\left\|\partial^{\ell}_t [u(t,\cdot)-u_k(t,\cdot)]\right\|_{H^{m-\ell-1}} \\
&\qquad\qquad\leq  
 C(T)\left(\sum_{\ell=0}^{m-1}\left\|\partial^{\ell}_t [u(0,\cdot)-u_k(0,\cdot)]\right\|_{H^{m-\ell-1}} + \int_0^T \left\|P[u_k](\tau,\cdot)\right\|_{L^2}d\tau \right)\ .
\end{align*}
The first term of the right hand side, which represents the difference in the initial data, can be estimated by \theo{InitialData} and the second term, which represents the evolution error, can be estimated by \lem{PuInTermsOfQ} to obtain
\begin{align*}
&\sum_{\ell=0}^{m-1}\left\|\partial^{\ell}_t [u(t,\cdot)-u_k(t,\cdot)]\right\|_{H^{m-\ell-1}} \\
&\qquad\qquad\leq  
 C(T)\left( \varepsilon^{\frac{k}{2}-m+1} + \sum_{j=1}^{J} \varepsilon^{\ell_j} \sup_{t\in[0,T]} \left\|\Op_{\alpha_j,g_j,\eta}\right\|_{L^2} \right) + \Ordo(\varepsilon^\infty)\ ,
\end{align*}
with $\ell_j\geq (k/2-m+1)$
and $\Op_{\alpha_j,g_j,\eta}$ satisfying (A1)-(A5), for small enough
$\eta$ when $k>1$.
Thus, using \theo{QL2norm}, we obtain
\begin{align*}
\sum_{\ell=0}^{m-1}\left\|\partial^{\ell}_t [u(t,\cdot)-u_k(t,\cdot)]\right\|_{H^{m-\ell-1}} \leq  
 C(T)\varepsilon^{\frac{k}{2}-m+1} \ ,
\end{align*}
which completes the proof for strictly hyperbolic PDEs. Since the wave equation
is a second order strictly hyperbolic PDE, applying the above estimate to
\eq{WaveEquationIntro}, we obtain for $t\in[0,T]$,
\begin{align*}
 \|u(t,\cdot)-u_k(t,\cdot)\|_E \leq \varepsilon\sum_{\ell=0}^{1}\left\|\partial^{\ell}_t [u(t,\cdot)-u_k(t,\cdot)]\right\|_{H^{1-\ell}} \leq  
 C(T)\varepsilon^{\frac{k}{2}} \ ,
\end{align*}
which completes the proof of \theo{ErrorEstWaveSchrod} for the wave equation.

For the Schr\"odinger equation \eq{SchrodingerEquationIntro}, applying the
well-posedness estimate given in \theo{wellPosednessAll} to the difference
between the true solution $u$ and the $k$-th order Gaussian beam superposition,
$u_k$, defined in \sect{SchrodingerEq}, we obtain for $t\in[0,T]$,
\begin{align*}
    \|u_k(t,\cdot)-u(t,\cdot)\|_{L^2} \leq \|u_k(0,\cdot)-u(0,\cdot)\|_{L^2} + \frac{1}{\varepsilon} \int_{0}^{T} \|P^\varepsilon[u_k](\tau, \cdot)\|_{L^2}d\tau \ .
\end{align*}
The initial data part of the right hand side can be estimated by \theo{InitialData} to obtain $\|u(0,\cdot)-u_k(0,\cdot)\|_{L^2} \leq C\varepsilon^{\frac{k}{2}}$. With the help of \lem{PuInTermsOfQ}, we can estimate the second part of the right hand side as
\begin{align*}
\frac{1}{\varepsilon} \int_{0}^{T} \|P^\varepsilon[u_k](\tau, \cdot)\|_{L^2}d\tau & \leq  \sum_{j=1}^{J} 
\varepsilon^{\ell_j -1} \sup_{t \in [0, T]} \|\Op_{\alpha_j,g_j,\eta}\|_{L^2} +\Ordo(\varepsilon^\infty) \ ,
\end{align*}
with $\ell_j\geq (k/2+1)$
and, as above, $\Op_{\alpha_j,g_j,\eta}$ satisfying (A1)-(A5), for small enough
$\eta$ when $k>1$.  Again, using \theo{QL2norm} and combining, we obtain,
\begin{align*}
 \|u_k(t,\cdot)-u(t,\cdot)\|_{L^2} & \leq C\varepsilon^\frac{k}{2} + \varepsilon^{\frac{k}{2}}\sum_{j=1}^{J}C_{\alpha}(T) +\Ordo(\varepsilon^\infty) \leq C(T)\varepsilon^\frac{k}{2} \ . 
\end{align*}
Thus, the proof of \theo{ErrorEstWaveSchrod} is complete.


\section{Norm Estimates of $\Op_{\alpha,g,\eta}$}\lbsec{NormEstQ}

In this section we prove \theo{QL2norm}. We follow the ideas in
\cite{BougachaEtal:09,RousseSwart:07} to relate the estimate of the oscillatory integral to the
operator norm, through the use of an adjoint operator. A key ingredient in
estimating the operator norm is the non-squeezing lemma (\lem{nonsqueezing}), which
allows us to obtain a dimensionally independent estimates for the oscillatory
integral operator.

\subsection{Operator Norm Estimates of $\Op_{\alpha,g,\eta}$}\lbsec{NormEstL2}

We let $\Op_{\alpha,g,\eta}^*$ be the adjoint operator and consider the
squared expression,
\begin{align}
\lbeq{Tu}
&(\Op_{\alpha,g,\eta}^*\Op_{\alpha,g,\eta}u)(t,z) \notag \\
&\qquad= \varepsilon^{-n-|\alpha|} \int_{\Real^n\times K_0} u(z') 
      e^{i\phi(t,y-x(t;z);z')/\varepsilon}\overline{e^{i\phi(t,y-x(t;z);z)/\varepsilon}}
      g(t,y;z')\overline{g(t,y;z)} \notag \\ 
      & \qquad\qquad\qquad\qquad \times (y-x(t;z'))^\alpha (y-x(t;z))^\alpha \rho_\eta(y-x(t;z'))\rho_\eta(y-x(t;z)) dy dz' \notag \\
      &\qquad:= \varepsilon^{-n-|\alpha|} \int_{K_0} I_{\alpha,g}^\varepsilon(t,z,z')u(z')dz',
\end{align}
where
\begin{align*}
 I_{\alpha,g}^\varepsilon(t,z,z') &= \int_{\Real^n} e^{i\phi(t,y-x(t;z);z')/\varepsilon}\overline{e^{i\phi(t,y-x(t;z);z)/\varepsilon}} g(t,y;z')\overline{g(t,y;z)} \\
    &\qquad\qquad  \times  (y-x(t;z'))^\alpha (y-x(t;z))^\alpha \rho_\eta(y-x(t;z'))\rho_\eta(y-x(t;z)) dy \\
  &=  \int_{\Real^n} e^{i\psi(t,y,z,z')/\varepsilon} 
       g(t,y+\bar{x};z')\overline{g(t,y+\bar{x};z)} \\
    &\qquad\qquad  \times  \left(y-\Delta x\right)^\alpha\left(y+\Delta x\right)^\alpha
       \rho_\eta\left(y-\Delta x\right)\rho_\eta\left(y+\Delta x\right) dy\ ,
\end{align*}
after a change of variables and 
\begin{eqnarray*}
 \bar{x} =\bar{x}(t,z,z')&:=&\frac{x(t;z)+x(t;z')}{2}\ ,\\
 \Delta x =\Delta x(t,z,z')&:=&\frac{x(t;z)-x(t;z')}{2}\ , \\
 \psi(t,y,z,z') &:=& \phi(t,y+\Delta x;z') -\overline{\phi(t;y-\Delta x;z)}\ .
\end{eqnarray*}
This symmetrization will simplify expressions later on.

Recall Schur's lemma:
\begin{lemma}[Schur]
For integrable kernels $K(x,y)$,
\begin{align*}
  \left\|\int K(x,y)u(x)dx \right\|_{L^2}^2\leq \left(\sup_x \int|K(x,y)|dy \right)\left(
  \sup_y \int|K(x,y)|dx \right)||u||_{L^2}^2\ .
\end{align*}
\end{lemma}
Using Schur's lemma, we can now deduce that
\begin{align*}
  ||\Op_{\alpha,g,\eta}||_{L^2}^2 &= 
      \sup_{w\in L_2(K_0)}\frac{\langle w,\Op_{\alpha,g,\eta}^*\Op_{\alpha,g,\eta}w\rangle}{||w||_{L^2}^2}
    \leq \sup_{w\in L_2(K_0)} \frac{||\Op_{\alpha,g,\eta}^*\Op_{\alpha,g,\eta}w||_{L^2}}{||w||_{L^2}} \\
   &\leq \varepsilon^{-n-|\alpha|}\left(\sup_{z\in K_0} \int_{K_0} |I_{\alpha,g}^\varepsilon(t,z,z')| dz'\right)^\frac12 \left(\sup_{z'\in K_0} \int_{K_0} |I_{\alpha,g}^\varepsilon(t,z,z')| dz\right)^\frac12\\
   &\leq \varepsilon^{-n-|\alpha|}\left(\sup_{z\in K_0} \int_{K_0} |I_{\alpha,g}^\varepsilon(t,z,z')| dz'\right)
\end{align*}
 upon noting that $|I_{\alpha,g}^\varepsilon(t,z,z')|=|{I_{\alpha,g}^\varepsilon(t,z',z)}|$.

Before continuing, we need some utility results.

\subsubsection{Utility results}

We will prove a few general results that will be useful in the proof of \theo{QL2norm}. 

\begin{lemma}[Phase estimate]\lblem{PhaseEst} Let $\eta$ be the same as in assumption (A4). Then, under the assumptions (A2)--(A4), $t\in[0,T]$, and $y$ such that $|y\pm\Delta{x}|\leq 2\eta$ (or all $y$ if $\eta=\infty$), we have:
\begin{itemize}
\item For all $z,z'\in K_0$, there exists a constant $\delta$ independent of $t$ such that
$$
\Im \psi\left(t,y,z,z'\right) \geq\  \frac12\delta\left[\left|y+\Delta x\right|^2+\left|y-\Delta x\right|^2\right]\ =\ \delta|y|^2+\frac14\delta|x(t;z)-x(t;z')|^2\ .
$$
\item For $|x(t;z)-x(t;z')|\leq \theta |z-z'|$, 
\begin{align*}
 \inf_{y\in\Omega(t,\mu)}|\nabla_y\psi(t,y,z,z')| \geq C(\theta,\mu)|z-z'|\ ,
\end{align*}
where $\Omega(t,\mu)=\{y : |y-\Delta x|\leq 2\mu \mbox{ and } |y+\Delta x|\leq 2\mu
\}$ and $C(\theta,\mu)$ is independent of $t$ and positive if $\theta$ and
$\mu<\eta$ are sufficiently small.
\end{itemize}
\end{lemma}
\begin{proof}
By assumption (A4), there exists a constant $\delta$ independent of $t$ such that
\begin{align*}
\Im\psi\left(t,y,z,z'\right) 
  &=   \Im\phi(t,y+\Delta x;z') + \Im\phi(t,y-\Delta x;z) 
  \geq\delta \left(|y+\Delta x|^2 + |y-\Delta x|^2 \right) \\
  &=   \delta\left[\left|y+\frac{x-x'}{2}\right|^2+\left|y-\frac{x-x'}{2}\right|^2\right]  = 2\delta|y|^2 + \frac12\delta|x-x'|^2\ .
\end{align*}
For convenience, we divide by $1/2$ to eliminate the factor in front of $\delta|y|^2$. 

For the second result, we have
\begin{align*}
 |\nabla_y\psi(t,y,z,z')| 
  \geq{}& |\Re\nabla_y\psi(t,y,z,z')| \\
  ={}& |\Re\nabla_y\phi(t,y+\Delta x;z')-\Re\nabla_y\phi(t,y-\Delta x;z)| \\
  ={}& \Bigl|\Re\nabla_y\phi(t,0;z')-\Re\nabla_y\phi(t,0;z)\\
  & +
   \Re\nabla_y\phi(t,y;z') - \Re\nabla_y\phi(t,0;z') - [\Re\nabla_y\phi(t,y;z) - \Re\nabla_y\phi(t,0;z)] \\
  & +\Re\nabla_y\phi(t,y+\Delta x;z')-\Re\nabla_y\phi(t,y;z')  \\
  & -[\Re\nabla_y\phi(t,y-\Delta x;z)-\Re\nabla_y\phi(t,y;z)]\Bigr|\\
  \geq{}& |\Re\nabla_y\phi(t,0;z')-\Re\nabla_y\phi(t,0;z)|\\
  & -
  \left| \left(\nabla_y\phi(t,y;z') - \nabla_y\phi(t,0;z')\right) - 
         \left(\nabla_y\phi(t,y;z) - \nabla_y\phi(t,0;z)\right) \right|\\
  & -|\nabla_y\phi(t,y+\Delta x;z')-\nabla_y\phi(t,y;z')|  \\
  & -|\nabla_y\phi(t,y-\Delta x;z)-\nabla_y\phi(t,y;z)|\\
  =:{}& E_1 - E_2 - E_3^+-E_3^-.
\end{align*}
Using assumption (A3) we have for $E_1$,
\begin{align*}
  E_1&=|\Re\nabla_y\phi(t,0;z')-\Re\nabla_y\phi(t,0;z)|=
  |\nabla_y\phi(t,0;z')-\nabla_y\phi(t,0;z)|\\
  &\geq C|z-z'| - |x-x'|,
\end{align*}
where $x=x(t;z)$ and $x'=x(t;z')$.
To estimate $E_2$, we
first note that on $\Omega(t,\mu)$,
\begin{align*}
 |y|=\left|\frac12 y - \frac12\Delta x + \frac12 y +\frac12\Delta x \right| \leq \frac12\left( |y-\Delta x| + |y+\Delta x|\right)\leq 2 \mu \ .
\end{align*}
Then, by the Fundamental Theorem of Calculus and the smoothness assumption (A2), for $y\in\Omega(t,\mu)$ we have,
\begin{align*}
 E_2 = \left| \int_0^1 \left[ \partial_y^2\phi(t,sy;z') - \partial_y^2\phi(t,sy;z) \right] y ds \right|\leq C|z-z'||y| \leq C_1\mu|z-z'| \ ,
\end{align*}
with $C_1$ independent of $t\in[0,T]$. Similarly, with $y\in\Omega(t,\mu)$ and
$C_2$ independent of $t\in[0,T]$ and $s\in K_0$,
\begin{align*}
 \left|\nabla_y\phi(t,y\pm\Delta x;s) - \nabla_y\phi(t,y;s)\right| \leq C_2 |\Delta x| =\frac12 C_2|x-x'|\ ,
\end{align*}
which shows that $E_3^++E_3^-\leq C_2|x-x'|$.
 Using these estimates for the case $|x-x'|\leq \theta |z-z'|$ we then obtain
\begin{align*}
 |\nabla_y\psi(t,y,z,z')| 
  &\geq C|z-z'| - |x-x'| - C_1\mu|z-z'| - C_2|x-x'|  \\
  &\geq C|z-z'| - (1+C_2)\theta|z-z'| - C_1\mu|z-z'| \\
  &=: C(\theta,\mu)|z-z'|\ ,
\end{align*}
where $C(\theta,\mu)$ is independent of $t\in[0,T]$ and positive if $\theta$ and $\mu$ are small enough.
\end{proof}

Next we have a version of the non-stationary phase lemma.

\begin{lemma}[Non-stationary phase lemma]\lblem{statphase}
Suppose that $u(y;\zeta)\in C_0^\infty(D\times Z)$, where $D$ and $Z$ are
compact sets and ${\psi(y;\zeta)\in C^\infty(O)}$ for some open neighborhood $O$
of $D\times Z$. If $\nabla_y\psi$ never vanishes in $O$, then for any
$K=0,1,\ldots$,
\begin{align*}
   \left|\int_D u(y;\zeta)e^{i\psi(y;\zeta)/\varepsilon}dy \right|
   \leq C_K \varepsilon^K \sum_{|\alpha|\leq K}\int_D \frac{|\partial^{\alpha}u(y;\zeta)|}{|\nabla_y\psi(y;\zeta)|^{2K-|\alpha|}} e^{-\Im \psi(y;\zeta)/\varepsilon}dy\ ,
\end{align*}
where $C_K$ is a constant independent of $\zeta$.
\end{lemma}
\begin{proof}
This is a classical result. A proof can be obtained by modifying the proof of
Lemma~7.7.1 of \cite{HormanderI}. However, we omit the details for the sake of
brevity.
\end{proof}

With this lemma in hand we can estimate $|I^\varepsilon_{\alpha,g}(t,z,z')|$.

\begin{lemma}\lblem{Ilemma}
Under the assumptions (A2)--(A5), for any $K=0,1,\ldots$,
fixed $0<\mu<\eta\leq\infty$, $s>0$ 
and $t\in[0,T]$, there are constants $C_K$ and $C_s$ independent of $t$ such that
\be{Ilemmares}
|I^\varepsilon_{\alpha,g}(t,z,z')| \leq C_K \varepsilon^{n/2+|\alpha|} 
                 \frac{\exp\left(-\frac{\delta|\Delta x|^2}{\varepsilon}\right)}
                {1+\inf_{y\in\Omega(t,\mu)}|\nabla_y\psi(t,y,z,z')/\sqrt\varepsilon|^{K}} 
                + C_s\varepsilon^s\ ,
\ee
where $\Omega(t,\mu)=\{y : |y-\Delta x|\leq 2\mu \mbox{ and } |y+\Delta x|\leq 2\mu \} \subseteq \{y : |y|\leq 2\mu\}$ is a compact set.
\end{lemma}
\begin{proof}
By the definition of $I^\varepsilon_{\alpha,g}(t,z,z')$, we have 
\begin{align*}
 I^\varepsilon_{\alpha,g}(t,z,z') &= \int_{\Real^n} e^{i\psi(t,y,z,z')/\varepsilon} g(t,y+\bar{x};z') \overline{g(t,y+\bar{x};z)} \\ 
 & \qquad\qquad\qquad\qquad \times (y-\Delta x)^\alpha (y+\Delta x)^\alpha  \rho_\eta(y+\Delta x) \rho_\eta(y-\Delta x)dy \\
&= \int_{\Omega(t,\mu)} e^{i\psi(t,y,z,z')/\varepsilon} g(t,y+\bar{x};z') \overline{g(t,y+\bar{x};z)} \\ 
 & \qquad\qquad\qquad\qquad \times (y-\Delta x)^\alpha (y+\Delta x)^\alpha  \rho_\eta(y+\Delta x) \rho_\eta(y-\Delta x)dy \\
&\quad+\int_{\Omega(t,\eta)\setminus\Omega(t,\mu)} e^{i\psi(t,y,z,z')/\varepsilon} g(t,y+\bar{x};z') \overline{g(t,y+\bar{x};z)} \\
 & \qquad\qquad\qquad\qquad \times (y-\Delta x)^\alpha (y+\Delta x)^\alpha  \rho_\eta(y+\Delta x) \rho_\eta(y-\Delta x)dy\\
&=: I_1 + I_2.
\end{align*}
The integral $I_1$ will correspond to the first part of the right hand side of the estimate in the lemma and $I_2$ to the second part. We begin estimating $I_1$.
 By \lem{PhaseEst} and (A5), for a fixed $t$, we compute,
\begin{align*}
 \left|I_1\right|
 &\leq C \int_{\Omega(t,\mu)} |y-\Delta x|^{|\alpha|} |y+\Delta x|^{|\alpha|}e^{-\delta(|y-\Delta x|^2+|y+\Delta x|^2)/\varepsilon}dy \ .
\end{align*}
 Now, using the estimate 
$ s^p e^{-as^2} \leq (p/e)^{p/2} a^{-p/2}e^{-as^2/2} \ ,$
with $p=|\alpha|$, $a=\delta/\varepsilon$ and $s=|y-\Delta x|$ or $|y+\Delta x|$, and continuing the estimate of $I_1$, we have for a constant, $C$, independent of $t$, $z$ and $z'$,
\begin{align*}
 \left|I_1\right| &\leq C\left(\frac{\varepsilon}{\delta}\right)^{|\alpha|} \int_{\Omega(t,\mu)}e^{-\frac{\delta}{2\varepsilon}(|y+\Delta x|^2+|y-\Delta x|^2)} \ dy \\
   &\leq C\left(\frac{\varepsilon}{\delta}\right)^{|\alpha|} \int_{\Omega(t,\mu)} e^{-\frac{\delta}{\varepsilon}|y|^2-\frac{\delta}{\varepsilon}|\Delta x|^2} \ dy
    \leq C\varepsilon^{n/2+|\alpha|} e^{-\frac{\delta}{\varepsilon}|\Delta x|^2} \ .
\end{align*}
Thus, we have proved the needed estimate for $I_1$ for the case $K=0$ as well as
the the case $K>0$ when $\inf_{y\in{\Omega(t,\mu)}}|\nabla_y\psi(t,y,z,z')|=0$.
Therefore, in the remainder of the proof we will consider the case $K\neq 0$ and
$\inf_{y\in{\Omega(t,\mu)}}|\nabla_y\psi(t,y,z,z')|\neq 0$. In this case,
\lem{statphase} can be applied to $I_1$ with $\zeta=(t,z,z')\in[0,T]\times K_0\times K_0$ to give,
\begin{align*}
 \left|I_1\right| &\leq C_K\varepsilon^K \sum_{|\beta|\leq K} \int_{\Omega(t,\mu)}\frac{\left|\partial^\beta_y\left[(y-\Delta x)^\alpha(y+\Delta x)^\alpha g'\overline{g}\rho^+_\eta\rho^-_\eta\right]\right|}{|\nabla_y\psi(t,y,z,z')|^{2K-|\beta|}}e^{-\Im\psi(t,y,z,z')/\varepsilon} dy \\
 &\leq C_K \sum_{|\beta|\leq K}\Bigg(\frac{\varepsilon^{|\beta|/2}}{\inf_{y\in{\Omega(t,\mu)}}|\nabla_y\psi/\sqrt{\varepsilon}|^{2K-|\beta|}}
\\ &\qquad\qquad\qquad\qquad\times\int_{\Omega(t,\mu)} \left| \partial^\beta_y\left[(y-\Delta x)^\alpha(y+\Delta x)^\alpha g'\overline{g}\rho^+_\eta\rho^-_\eta\right]\right|e^{-\Im\psi/\varepsilon} dy \Bigg) \\
 &\leq C_K \sum_{|\beta|\leq K}\frac{\varepsilon^{|\beta|/2}}{\nu(t,z,z')^{2K-|\beta|}}
\Bigg(\sum_{\substack{\beta_1+\beta_2=\beta\\ \beta_1\leq2\alpha}}\int_{\Omega(t,\mu)} \left| \partial^{\beta_1}_y\left[(y-\Delta x)^\alpha(y+\Delta x)^\alpha\right]\right|\\
&\qquad\qquad\qquad\qquad\qquad\qquad\qquad\qquad\qquad\qquad\times\left|\partial^{\beta_2}_y\left[ g'\overline{g}\rho^+_\eta\rho^-_\eta\right]\right|e^{-\Im\psi/\varepsilon} dy \Bigg) \ ,
\end{align*}
where $\rho_\eta^\pm = \rho_\eta(y\pm\Delta x)$,
$\nu(t,z,z')=\inf_{y\in{\Omega(t,\mu)}}|\nabla_y\psi(t,y,z,z')/\sqrt{\varepsilon}|$
and $C_K$ is independent of $t$, $z$ and $z'$. By assumption (A5) and since
$\rho_\eta$ is uniformly smooth and $t$, $z$, $z'$ vary in a compact set,
$\left|\partial^{\beta_2}_y\left[
g'\overline{g}\rho^+_\eta\rho^-_\eta\right]\right|$ can be bounded by a constant
independent of $y$, $t$, $z$ and $z'$. We estimate the other term as follows,
\begin{align*}
 \left| \partial_y^{\beta_1} \left[ (y-\Delta x)^\alpha (y+\Delta x)^\alpha \right]\right| 
  &\leq C  \sum_{\substack{\beta_{11}+\beta_{12}=\beta_1\\ \beta_{11},\beta_{12}\leq\alpha}}\left| (y-\Delta x)^{\alpha-\beta_{11}}(y+\Delta x)^{\alpha-\beta_{12}}\right| \\
 &\leq C \sum_{\substack{\beta_{11}+\beta_{12}=\beta_1\\ \beta_{11},\beta_{12}\leq\alpha}}|y-\Delta x|^{|\alpha|-|\beta_{11}|}\ |y+\Delta x|^{|\alpha|-|\beta_{12}|}\ .
\end{align*}
Now, using the same argument as for the $K=0$ case, we have
\begin{align*}
& \int_{\Omega(t,\mu)} \left| \partial^{\beta_1}_y\left[(y-\Delta x)^\alpha(y+\Delta x)^\alpha\right]\right|\left|\partial^{\beta_2}_y\left[ g'\overline{g}\rho^+_\eta\rho^-_\eta\right]\right|e^{-\Im\psi/\varepsilon} dy \\
&\quad\qquad\leq C \sum_{\substack{\beta_{11}+\beta_{12}=\beta_1\\ \beta_{11},\beta_{12}\leq\alpha}}\int_{\Omega(t,\mu)} |y-\Delta x|^{|\alpha|-|\beta_{11}|}\ |y+\Delta x|^{|\alpha|-|\beta_{12}|}e^{-\Im\psi/\varepsilon} dy \\
&\quad\qquad\leq C(\beta_2) \varepsilon^{\frac{n + |\alpha|-|\beta_{11}|+|\alpha|-|\beta_{12}|}{2}}e^{-\frac{\delta}{2\varepsilon}|\Delta x|^2} = C(\beta_2) \varepsilon^{n/2 + |\alpha|-|\beta_1|/2}e^{-\frac{\delta}{\varepsilon}|\Delta x|^2} \ ,
\end{align*}
and consequently,
\begin{align*}
\left|I_1\right|
& \leq C_K \sum_{|\beta|\leq K}\frac{\varepsilon^{|\beta|/2}}{\nu(t,z,z')^{2K-|\beta|}}\sum_{\substack{\beta_1+\beta_2=\beta\\ \beta_1\leq2\alpha}}C(\beta_2) \varepsilon^{n/2 + |\alpha|-|\beta_1|/2}e^{-\frac{\delta}{\varepsilon}|\Delta x|^2} \\
& \leq C_K\varepsilon^{n/2 + |\alpha|}e^{-\frac{\delta}{\varepsilon}|\Delta x|^2}\sum_{|\beta|\leq K}\frac{1}{\nu(t,z,z')^{2K-|\beta|}} \ . 
\end{align*}
Using the fact that $\left|I_1\right|$ will be bounded by the minimum of the $K=0$ and $K>0$ estimates, we have
\begin{align*}
\left|I_1\right|
& \leq C\varepsilon^{n/2 + |\alpha|}e^{-\frac{\delta}{\varepsilon}|\Delta x|^2}\min\left[1,\sum_{|\beta|\leq K}\frac{1}{\nu(t,z,z')^{2K-|\beta|}}\right]\ .
\end{align*}
Noting that for positive $a$, $b$, and $c$,
\begin{align*}
\min[a,b+c]\leq \min[a,b]+\min[a,c] \qquad \mbox{and}\qquad \min[1,1/a]\leq 2/(1+a)\ ,
\end{align*}
we have,
\begin{align*}
 \min\left[1,\sum_{|\beta|\leq K}\frac{1}{\nu(t,z,z')^{2K-|\beta|}}\right] 
 &\leq \sum_{|\beta|\leq K}\min\left[1,\frac{1}{\nu(t,z,z')^{2K-|\beta|}}\right] \\
 &\leq \sum_{|\beta|\leq K}\frac{2}{1+\nu(t,z,z')^{2K-|\beta|}} \leq C_K\frac{1}{1+\nu(t,z,z')^{K}} \ .
\end{align*}
This shows the $I_1$ contribution to the estimate \eq{Ilemmares}. It remains to
show the smallness of $I_2$. Indeed, since either $|y+\Delta x|>\mu$ or
$|y-\Delta x|>\mu$ on $\Omega(t,\eta)\setminus\Omega(t,\mu)$, we get in the same way
as for $I_1$ in the $K=0$ case,
\begin{align*}
 \left|I_2\right| &\leq C\left(\frac{\varepsilon}{\delta}\right)^{|\alpha|} \int_{\Omega(t,\eta)\setminus\Omega(t,\mu)}e^{-\frac{\delta}{2\varepsilon}(|y+\Delta x|^2+|y-\Delta x|^2)} \ dy \leq 
 C\left(\frac{\varepsilon}{\delta}\right)^{|\alpha|+\frac{n}{2}}
 e^{-\frac{\delta\mu}{2\varepsilon}}\leq C_s\varepsilon^s\ ,
\end{align*}
for any $s>0$ with $C_s$ independent of $t$. This concludes the proof of the lemma.
\end{proof}


\subsubsection{Proof of \theo{QL2norm}}\lbsec{L2est}


We now have all of the ingredients to complete the proof of \theo{QL2norm}.
We fix $t\in[0,T]$ and start with the estimate,
\begin{align*} 
  ||\Op_{\alpha,g,\eta}||_{L^2}^2 
   &\leq \varepsilon^{-n-|\alpha|}\left(\sup_z \int_{K_0} |I_{\alpha,g}^\varepsilon(t,z,z')| dz'\right) \ ,
\end{align*}
derived in the beginning of \sect{NormEstL2} and we turn our attention to estimating the integral of $|I^\varepsilon_{\alpha,g}(t,z,z')|$. 

We will use the shorthand notation $I^\varepsilon_D(t,z,z') \equiv \chi_D(z,z')I^\varepsilon_{\alpha,g}(t,z,z')$, where $\chi_D(z,z')$ is the characteristic function on a domain $D\subseteq K_0\times K_0$. We will consider two disjoint subsets of $K_0\times K_0$ given by 
\begin{align*}
 D_1(t,\theta) & = \{(z,z') : |x(t;z)-x(t;z')|\geq \theta|z-z'| \} \\
 D_2(t,\theta) & = \{(z,z') : |x(t;z)-x(t;z')|< \theta|z-z'| \} \ ,
\end{align*} 
where $\theta$ is the small parameter $\theta$ in \lem{PhaseEst}. Note that $D_1\cup D_2=K_0\times K_0$ and $D_1\cap D_2=\emptyset$. Thus,
\begin{align*}
 \int_{K_0}\left|I^\varepsilon_{\alpha,g}(t,z,z')\right| dz' = 
 \int_{K_0}\left|I^\varepsilon_{D_1}(t,z,z')\right| dz' +  \int_{K_0}\left|I^\varepsilon_{D_2}(t,z,z')\right| dz' \ .
\end{align*}
The set $D_1$ corresponds to the non-caustic region of the solution. There, we estimate $I^\varepsilon_{D_1}(t,z,z')$ by taking $K=0$ and $s=n+|\alpha|$ in \lem{Ilemma}.
\begin{align*}
\int_{K_0}\left|I^\varepsilon_{D_1}(t,z,z')\right| dz' 
 &\leq  C \varepsilon^{n/2+|\alpha|}\int_{K_0} e^{-\frac{\delta}{4\varepsilon}|x(t;z)-x(t;z')|^2} +\varepsilon^{n/2} dz'\\
 &\leq  C \varepsilon^{n/2+|\alpha|}\int_{K_0} e^{-\frac{\delta\theta^2}{4\varepsilon}|z-z'|^2}dz'+C|K_0|\varepsilon^{n+|\alpha|}\\
 &\leq  C \varepsilon^{n/2+|\alpha|}\int_0^\infty s^{n-1}e^{-\frac{\delta\theta^2}{4\varepsilon}s^2} ds +C\varepsilon^{n+|\alpha|}\\
&\leq C \varepsilon^{n+|\alpha|}\ .
\end{align*}
The set $D_2$ corresponds to the region near caustics of the solution. On $D_2$, we estimate $I^\varepsilon_{D_2}(t,z,z')$ using \lem{Ilemma} again
where we pick $\mu$ small enough to allow us to use also \lem{PhaseEst}.
Letting $R=\sup_{z,z'\in K_0}|z-z'|< \infty$ be the diameter of $K_0$
and $s=n+|\alpha|$, we compute
\begin{align*}
\int_{K_0}\left|I^\varepsilon_{D_2}(t,z,z')\right| dz' 
&\leq  C \varepsilon^{\frac{n}{2}+|\alpha|}\int_{K_0}\left( \frac{e^{-\frac{\delta}{4\varepsilon}|x(t;z)-x(t;z')|^2}}{1 + \inf_{y\in\Omega(t,\mu)} |\nabla_y\psi(t,y,z,z')/\sqrt{\varepsilon}|^K} +\varepsilon^{\frac{n}{2}}\right)dz'\\
&\leq  C \varepsilon^{\frac{n}{2}+|\alpha|}\int_{K_0} \frac{1}{1 + \left(\frac{C(\theta,\mu)|z-z'|}{\sqrt{\varepsilon}}\right)^K} dz'+ C\varepsilon^{n+|\alpha|}\\
&\leq  C \varepsilon^{\frac{n}{2}+|\alpha|}\int_0^R \frac{1}{1 + (C(\theta,\mu)s/\sqrt{\varepsilon})^K}s^{n-1} ds + C\varepsilon^{n+|\alpha|} \\
&\leq  C \varepsilon^{n+|\alpha|} \ ,
\end{align*}
if we take $K=n+1$.

Since all of the constants are independent of the fixed $t\in[0,T]$, by putting all of these estimate together we obtain
\begin{align*}
  ||\Op_{\alpha,g,\eta}||_{L^2}^2 
\leq \varepsilon^{-n-|\alpha|}\left(\sup_{z\in K_0} \int_{K_0} |I_{\alpha,g}^\varepsilon(t,z,z')| dz'\right) \leq C \ ,
\end{align*}
for all $t\in[0,T]$, which proves the theorem.


\section{Numerical study of convergence} \lbsec{Numerical}

In this section, we perform numerical convergence analyses to study the
sharpness of the theoretical estimates in this paper for the constant
coefficient wave equation with sound speed $c(y)=1$, for which $H(t,x,p)=\pm
|p|$. The ODEs that define the Gaussian beams are solved numerically using an
explicit Runge-Kutta $(4,5)$ method (MATLAB's {\tt ode45}). We use the fast
Fourier transform to obtain the ``exact solution'' and use it to determine the
error in the Gaussian beam solution. When the norms require it, we compute
derivatives via analytical forms rather than numerical differentiation.

\subsection{Single Gaussian Beams}

First, we study the convergence rate for a single Gaussian beam to show
the sharpness of the estimate proved in \cite{Ralston:82}. For $2$D, this
estimate states that for a single Gaussian beam $v_k(t,y)$,
\begin{align*}
  \| (\partial^2_t-c(y)^2\Delta) v_k(t,y) \|_{L^2} \leq C \varepsilon^{k/2-1/2} \ ,
\end{align*}
for $t\in[0,T]$. Using the well-posedness estimate for the wave equation, we
obtain the rescaled energy norm estimate
\begin{align*}
  || v_k(t,\cdot)-u(t,\cdot) ||_E &\leq || v_k(0,\cdot)-u(0,\cdot) ||_E + T\;\varepsilon \sup_{t\in[0,T]}||(\partial^2_t-c(y)^2\Delta) v_k(t,\cdot)||_{L^2} \ ,
\end{align*}
for $t\in[0,T]$, where $u$ is the exact solution to the wave equation. Taking
the initial conditions for $u$ to be the same as the $1$-st, $2$-nd and $3$-rd
order Gaussian beams at $t=0$ (modulo the cutoff function), we obtain the
asymptotic error estimate,
\begin{align*}
  \| v_k(t,\cdot)-u(t,\cdot) \|_E &\leq C\varepsilon^{k/2+1/2}\ .
\end{align*}
To test the sharpness of this estimate, we investigate the numerical convergence as follows. We let the Gaussian beam parameters for the $1$-st order Gaussian beam be given by
\begin{align*}
   x(0) &= \begin{bmatrix} 0\\ 0 \end{bmatrix} \ ,
  &p(0) &= \begin{bmatrix} -1 \\ 0 \end{bmatrix}\ ,
   &\phi_0(0) &= 0 \ , \\
  M(0) &= \begin{bmatrix} i & 0 \\ 0 & 2+i \end{bmatrix}\ ,
 & a_{0,0}(0) &= 1 \ ,
\end{align*}
and pick $H=-|p|$. The additional coefficients necessary for the Taylor
polynomials of phase and amplitudes for $2$-nd and $3$-rd order Gaussian beams
are all initially taken to be $0$. Note that even though the phase and
amplitudes for $1$-st, $2$-nd and $3$-rd order beams are the same at $t=0$,
their time derivatives will not be, as the ODEs for the higher order
coefficients are inhomogeneous. Thus, since the initial data for the exact
solution, $u$, matches the initial data for the Gaussian beam, $u$ depends on
the order of the beam. With this choice of parameters, we generate the $1$-st,
$2$-nd and $3$-rd order Gaussian beam solution at $t=\{0.5,1\}$. For $2$-nd and
$3$-rd order beams we additionally use a cutoff function with $\eta=1/10$. At
each $t$, we compute the rescaled energy norm of the difference $v_k-u$. The
asymptotic convergence as $\varepsilon\to 0$ is shown in
\fig{SingleBeamAsymptotics-ENorm}. We draw the attention of the reader to the
following features of the plots in \fig{SingleBeamAsymptotics-ENorm}:
\begin{enumerate}
 \item For $1$-st order beam: The error decays as $\varepsilon^1$.
 \item For $2$-nd and $3$-rd order beams: The error for larger $\varepsilon$ is
dominated by the error induced by the cutoff function. In this region, the error
decays exponentially fast. As $\varepsilon$ gets smaller, the error decays as
$\varepsilon^{3/2}$ for the $2$-nd order beam and $\varepsilon^2$ for the $3$-rd
order beam.
\end{enumerate}
The numerical results agree with the estimate given in \cite{Ralston:82} and,
thus, the estimate for single Gaussian beams is sharp.

\begin{figure}[ht!]
\begin{center}
\epsfig{file=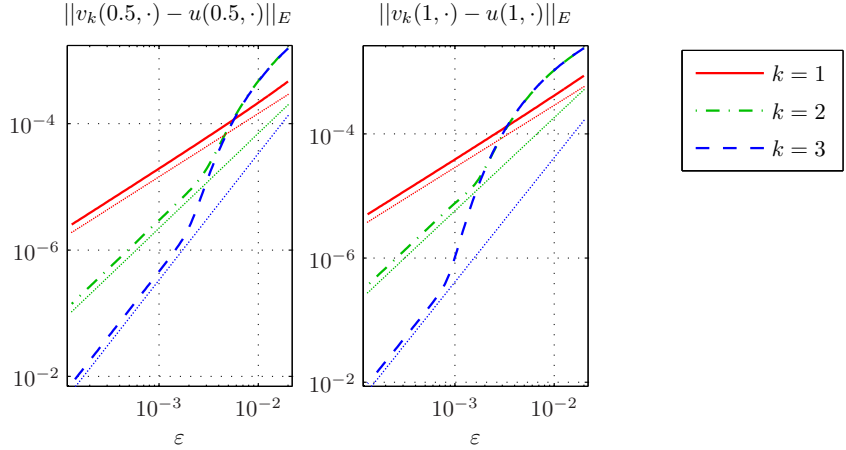}
\end{center}
\caption{Single Gaussian beams: Asymptotic behavior of $v_k-u$ in the rescaled
energy norm at $t=\{0.5,1\}$ for $k=\{1,2,3\}$ order beams. The results are
shown on log-log plots along with $c_1\varepsilon^{1}$, $c_2\varepsilon^{3/2}$,
and $c_3\varepsilon^{2}$ to help with the interpretation of the asymptotic
behavior. The asymptotic behavior agrees with the analytical estimates: $1$-st
order beam is $\Ordo(\varepsilon^{1})$, $2$-nd order beam is
$\Ordo(\varepsilon^{3/2})$, and $3$-rd order beam is $\Ordo(\varepsilon^{2})$.}
\lbfig{SingleBeamAsymptotics-ENorm}
\end{figure}



\subsection{Cusp Caustic}
We consider an example in $2$D that develops a cusp caustic. The initial data for $u$ at $t=0$ is given by
\begin{align*}
  u(0,y) &= e^{-10|y|^2} e^{i(-y_1+y_2^2)/\varepsilon}\ .
\end{align*}
Thus, the initial phase and amplitudes are given by
\begin{align*}
 \Phi(y) = -y_1 + y_2^2\ , \qquad
  A_{0,0}(y) = e^{-10|y|^2}, \qquad
  A_{0,1}(y) = 0 \ .
\end{align*}
For the initial data for $u_t(0,y)$, we take 
\begin{align*}
   u_t(0,y) &= \left( \frac{i}{\varepsilon}\Phi_t(y) \left[A_{0,0}(y)+\varepsilon A_{0,1}(y)\right] + \left[A_{0,0,t}(y)+\varepsilon A_{0,1,t}(y)\right]\right)e^{i\Phi(y)/\varepsilon} \ ,
\end{align*}
where the $\Phi_t$, $A_{0,0,t}$ and $A_{0,1,t}$ are obtained from the Gaussian
beam ODEs and various $y$ derivatives of $\Phi$, $A_{0,0}$ and $A_{0,1}$.
Specifically, we take $\Phi_t(y)=+|\nabla_y\Phi(y)|$ so that waves propagate in
the positive $y_1$ direction. As was shown in \cite{Tanushev:08}, this
particular example develops a cusp caustic at $t=0.5$ and two fold caustics for
$t>0.5$.

To form the Gaussian beam superposition solutions, it is enough to consider
Gaussian beams governed by $H=-|p|$, because of the initial data choice. We take
the initial Taylor coefficients to be
\begin{align*}
  x(0;z) &= \begin{bmatrix} z_1\\ z_2 \end{bmatrix} \ ,
 & p(0;z) &= \begin{bmatrix} -1 \\ 2z_2 \end{bmatrix}\ , \\
  \phi_0(0;z) &= -z_1+z_2^2\ , 
  &M(0;z) &= \begin{bmatrix} i & 0 \\ 0 & 2+i \end{bmatrix} \ ,\\
  \phi_\beta(0;z) &= 0, \ |\beta|= 3,4\ , 
  &a_{0,\beta}(0;z) &= \partial_y^\beta A_0(z), \ |\beta|=0,1,2 \ , \\
  a_{1,0} &= 0 \ , 
\end{align*}
where only the necessary parameters are used for each of the $1$-st, $2$-nd and
$3$-rd order beams. For $2$-nd and $3$-rd order beams, we additionally use a
cutoff function with $\eta=1/10$. We propagate the Gaussian beam solutions to
$t=\{0,0.25,0.5,0.75,1\}$. At each $t$, we compute the rescaled energy norm of
the difference $u_k-u$. The asymptotic convergence as $\varepsilon\to 0$ is
shown in \fig{ParabolicPhaseAsymptotics-ENorm}. 

\begin{figure}[h!]
\begin{center}
\epsfig{file=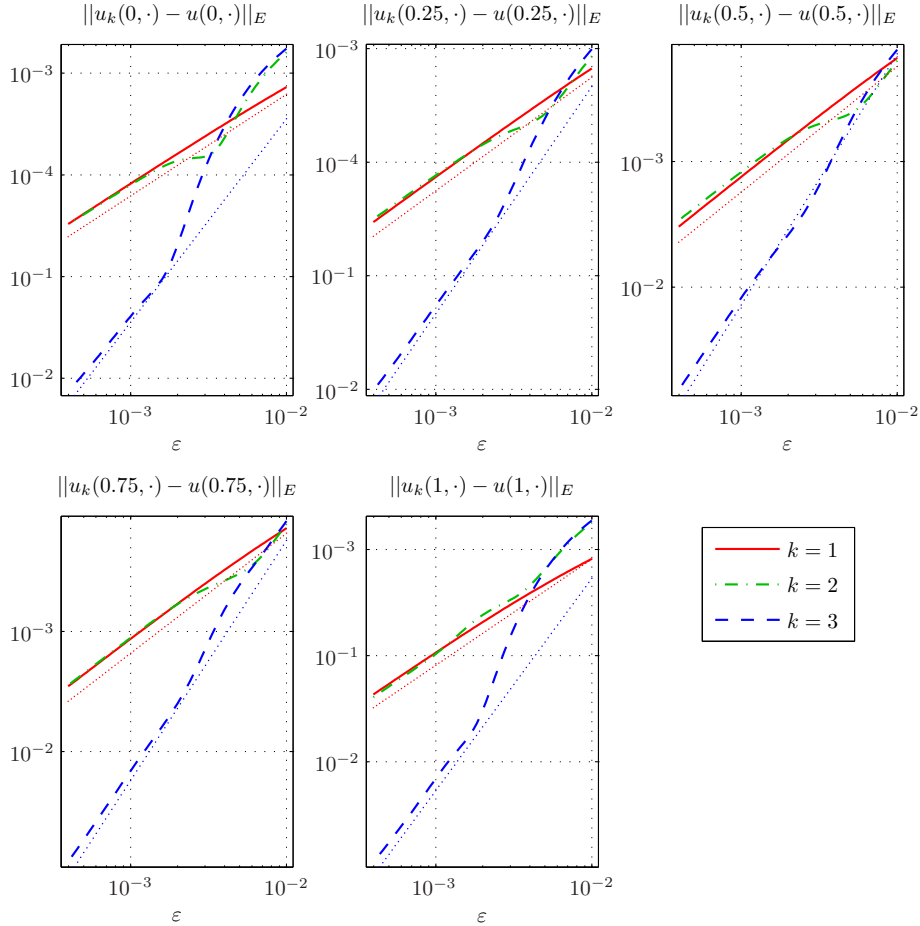}
\end{center}
\caption{Superpositions of Gaussian beams for cusp caustic: Asymptotic behavior
of ${u_k-u}$ in the rescaled energy norm at $t=\{0,0.25,0.5,0.75,1\}$ for
$k=\{1,2,3\}$ order Gaussian beam superpositions. The results are shown on
log-log plots along with $c_1\varepsilon^{1}$ and $c_2\varepsilon^{2}$ to help
with the interpretation of the asymptotic behavior.}
\lbfig{ParabolicPhaseAsymptotics-ENorm}
\end{figure}

We draw the attention of the
reader to the following features in the plots:
\begin{enumerate}
 \item At all $t$:
  \begin{enumerate}
    \item The asymptotic behavior of $1$-st and $2$-nd order solutions is the same and of order $\Ordo(\varepsilon^{1})$.
    \item The asymptotic behavior of $3$-rd order solution is of order $\Ordo(\varepsilon^{2})$.
    \item For large value of $\varepsilon$ the error for $2$-nd and $3$-rd order solutions is dominated by the error induced by the cutoff function. In this region the error decays exponentially.
    \item For midrange values of $\varepsilon$, $2$-nd order solutions experience a fortuitous error cancellation. This is due to the cutoff function. Changing the cutoff radius $\eta$ shifts this region.
  \end{enumerate}
  \item At $t=0.5$ and $t>0.5$, the asymptotic behavior of the error is unaffected by the cusp and fold caustics, respectively. 
\end{enumerate}

For this example we note that the convergence rate for odd beams $k\in\{1,3\}$
is in fact half an order better than our theoretical estimates. This improvement
was also observed and proved for a simplified setting in
\cite{MotamedRunborg:09}, where the analysis shows that the gain is due to error
cancellations between adjacent beams; it is therefore not present for single
Gaussian beams. The result in \cite{MotamedRunborg:09} is for the Helmholtz
equation and only concerns the pointwise error away from caustics. Here the
numerical results indicate that the same improvement appears for the wave
equation in the energy norm even when caustics are present. We conjecture that
this gain of convergence is due to error cancellations as well and that it is
present for all Gaussian beam superpositions with beams of odd order $k$.
Moreover, we conjecture that the theoretical result is sharp for beams of
even order $k$, giving us an optimal error estimate in the energy norm of
$\Ordo(\varepsilon^{\lceil k/2\rceil})$ for all $k$.

\section{Concluding Remarks}

Gaussian beams are asymptotically valid high frequency solutions to strictly
hyperbolic PDEs concentrated on a single curve through the physical domain. They
can also be constructed for the Schr\"odinger equation. Superpositions of
Gaussian beams provide a powerful tool to generate more general high frequency
solutions. In this work, we establish error estimates of the Gaussian beam
superposition for all strictly hyperbolic PDEs and the Schr\"odinger equation.
Our study gives the surprising conclusion that even if the superposition is done
over physical space, the error is still independent of the number of dimension
and of the presence of caustics. Thus, we improve upon earlier results by Liu
and Ralston \cite{LiuRalston:09,LiuRalston:10}.

\section*{Acknowledgement}
The authors would like to thank James Ralston and Bj\"orn Engquist for many
helpful discussions. HL's research was partially supported by the National
Science Foundation under Kinetic FRG grant No. DMS 07-57227 and grant No. DMS
09-07963. NMT was partially supported by the National Science Foundation
under grant No. DMS-0914465 and grant No. DMS-0636586 (UT Austin RTG).

\bibliographystyle{plain}
\bibliography{references}

\end{document}